\documentclass[11pt]{article}
\usepackage[usenames]{color}
\usepackage[colorlinks=true,pagebackref,hyperindex,citecolor=green,linkcolor=red]{hyperref}
\usepackage{amsmath,amsfonts,amsthm}
\usepackage{syntonly}
\usepackage[mathcal]{euscript}
\def\BB{{\mathbb B}}

\def\HH{{\mathbb H}}

\def\KK{{ K}}

\def\NN{{\mathbb N}}
\def\PP{{\mathbb P}}

\def\cG{{\cal G}}

\def\cO{{\cal O}}

\def\cK{{\cal K}}

\def\cO{{\cal O}}

\newtheorem{lema}{Lemma}[section]
\newtheorem{prop}{Proposition}[section]
\newtheorem{Cor}{Corolary}[section]
\newtheorem{Teo}{Theorem}[section]
\newtheorem{Rem}{Remark}[section]

\DeclareMathOperator{\tot}{tot}

\DeclareMathOperator{\coker}{coker}
\DeclareMathOperator{\Ann}{Ann}

\DeclareMathOperator{\Ker}{Ker}
\DeclareMathOperator{\Imag}{Im}
\DeclareMathOperator{\Hom}{Hom}

\newcommand{\FlechaDer}[1]{\stackrel{#1}{\longrightarrow}}
\newcommand{\FlechaIzq}[1]{\stackrel{#1}{\longleftarrow }}
\newcommand{\TAU}[1]{\bigl( \begin{smallmatrix}  c_{{#1}1} \cr c_{{#1}2}\end{smallmatrix} \bigr)}

\begin{document}
\title{On the Hyperhomology of the Small Gobelin
in Codimension 2}
\author{X.\,G\'omez-Mont\thanks{Supported by   CONACYT  134081, M\'exico.}\ \ \   and L. N\'u\~nez-Betancourt\thanks{Supported by   CONACYT  207063, M\'exico.} }

\maketitle
\begin{abstract}
Given a zero-dimensional Gorenstein algebra $\mathbb{B}$ and two syzygies between two elements $f_1,f_2\in\BB$, one constructs a double complex of $\mathbb{B}$-modules, ${\cal G}_\mathbb{B},$  called the small Gobelin. We describe an inductive procedure to construct the even and odd hyperhomologies of this complex. For high degrees, the difference $\dim \mathbb{H}_{j+2}({\cal G}_\mathbb{B}) - \dim\mathbb{H}_j({\cal G}_\mathbb{B})$  is constant, but possibly with a different value for even and odd degrees. We describe two flags of ideals in  $\mathbb{B}$ which codify the above differences of dimension. The motivation to study this double complex comes from understanding the tangency condition between a vector field and a complete intersection, and invariants constructed in the zero locus of the vector field $\hbox{Spec}(\mathbb{B})$.
\end{abstract}

\section*{Introduction}

Let $\KK$ be a field
of characteristic $0$
and  $\cO$ a local Noetherian $\KK$-algebra.
Consider a matrix identity over $\cO$

\begin{equation}
\begin{pmatrix}\varphi_{11}  \cdots \varphi_{1N}\cr
\vdots \ddots \vdots \cr
\varphi_{\ell 1}  \cdots \varphi_{\ell N}\cr
\end{pmatrix}\label{(0.1)}
\begin{pmatrix}X_1\cr
 \vdots \cr
X_N\cr
\end{pmatrix}
=
\begin{pmatrix}c_{11}&  \cdots & c_{1k}\cr
\vdots &\ddots& \vdots \cr
c_{\ell 1} & \cdots& c_{\ell k}\cr
\end{pmatrix}
\begin{pmatrix}f_1\cr
\vdots \cr
f_k\cr
\end{pmatrix}
\end{equation}
and write this equation  as
\begin{equation}
\varphi X = c f.
\label{E0}
\end{equation}
In particular if $\varphi$ is the derivative $Df$ of a differentiable function $f$, the identity expresses the
fact that the vector field $X$ is tangent to the variety defined by $f=0$.

Bothmer, Ebeling and the first author have constructed in \cite{BEG}
from the matrix identity (\ref{E0})
a double complex of free $\cO$-modules,
called the Gobelin, obtained by weaving Buchsbaum-Eisenbud and Koszul complexes
as strands, and that under
suitable hypothesis can be used to
compute the homology of the complex
of $\frac{\cO}{\Imag(f^*)}$-modules:

\begin{equation}
\label{PrimerComplejo}
 0  \FlechaIzq{}
 \frac{\cO}{\Imag(f^*)}   \FlechaIzq{X^*}
 \frac{\cO^{\oplus N}}{\Imag(\varphi^*)
  } \otimes_K \frac{\cO}{\Imag(f^*)}  \FlechaIzq{X^*}
 \frac{\Lambda^2\cO^{ \oplus N}}{\Lambda^2\Imag(\varphi^*)}
\otimes_K \frac{\cO}{\Imag(f^*)}
    \FlechaIzq{X^* }
  \cdots
\end{equation}
Similarly, they also constructed in \cite{BEG} another double complex $\cG^\ell_\BB$,
called the small Gobelin (see \cite{BEG} or Figure \ref{Gobelin22}), which is
quasi-isomorphic to the Gobelin, but formed
with free $ \BB:=\frac{\cO}{\Imag(X^*)} $-modules.
The homology of (\ref{PrimerComplejo})
has certain periodicity properties when $k=\ell=1$ (see \cite{G}),
and the objective of this work is to use the small
Gobelin to set up an inductive procedure on $\ell$ to see
how  these periodicity properties generalize.
In this paper, we carry this out for $k=2$ and $\ell=1,2$.
\vskip5mm

If $\cG^{\ell-1}_\BB$ denotes the small Gobelin formed with
the equation (\ref{(0.1)}) by deleting the last row on the matrices
in the equation, we obtain in Theorem \ref{ExactaGobelinos} an exact sequence
of double complexes
$$
0\FlechaDer{}\cG^{\ell-1}_\BB\FlechaDer{i}
 \cG^{\ell}_\BB\FlechaDer{\sigma^*}\cG^{\ell}_\BB (-1,-1)\FlechaDer{} 0
$$
giving rise to a long exact sequence of hyperhomology groups

$$
\ldots \FlechaDer{} \HH_j(\cG^{\ell-1}_\BB )\FlechaDer{\iota }
\HH_j( \cG^{\ell}_\BB )\FlechaDer{\sigma^*}
\HH_{j-2}( \cG^{\ell}_\BB  )\FlechaDer{\partial}
\HH_{j-1}(\cG^{\ell-1}_\BB ) \FlechaDer{} \dots.
$$
This long exact sequence
relates the hyperhomology groups  $\HH_j(\cG^{\ell}_\BB )$.
 with
$\HH_{j-2}(\cG^{\ell}_\BB )$
with kernels
and cokernels in the hyperhomology groups of $\cG^{\ell-1}_\BB$,
and this is the basis of our inductive procedure.
\vskip 5mm

From Section  \ref{Sec3} onwards, we restrict to the case where
 $\cO$ is an $N$ dimensional Gorenstein $\KK$-algebra  (in particular a local regular ring or the ring of power
 series over $\KK$ (formal or convergent) or a complete intersection),
assume given a matrix identity (\ref{(0.1)}) with $k=\ell=2$,
that $X_1,\ldots,X_N$ is an $\cO$-regular sequence and denote by
$\BB:=\frac{\cO}{\Imag(X^*)}$. It is a 0 dimensional Gorenstein ring, which is a finite dimensional  $\KK$-vector space.
The relation (\ref{(0.1)}) becomes in $\BB$:
\begin{equation}
\label{syzygyreduced}
\begin{pmatrix} c_{11}&c_{12}\cr
c_{21}&c_{22}\cr
\end{pmatrix}
\begin{pmatrix} f_1\cr
f_2\cr
\end{pmatrix} = \begin{pmatrix} 0\cr
0\cr
\end{pmatrix}
\end{equation}

The total complex $\tot(\cG^{2}_\BB)$ associated to the small Gobelin $\cG^{2}_\BB$ is
\\
\begin{equation}
\label{hyperhomologysequenceG3}
0 \xleftarrow{}
\BB^1
\xleftarrow{\left(
\begin{smallmatrix}{{f}}_{1}&
{{f}}_{{2}}\\
\end{smallmatrix}
\right)}
\BB^2
\xleftarrow{\left(
\begin{smallmatrix}{-{{f}}_{{2}}}&
{{c}}_{11}&
{{c}}_{{2}1}\\
{{f}}_{1}&
{{c}}_{1{2}}&
{{c}}_{{2}{2}}\\
\end{smallmatrix}
\right)}
\BB^3
\xleftarrow{\left(
\begin{smallmatrix}{-{{c}}_{1{2}}}&
{-{{c}}_{{2}{2}}}&
{{c}}_{11}&
{{c}}_{{2}1}\\
{{f}}_{1}&
0&
{{f}}_{{2}}&
0\\
0&
{{f}}_{1}&
0&
{{f}}_{{2}}\\
\end{smallmatrix}
\right)}
\BB^4
\xleftarrow{\left(
\begin{smallmatrix}{-{{f}}_{{2}}}&
0&
{{c}}_{11}&
{{c}}_{{2}1}&
0\\
0&
{-{{f}}_{{2}}}&
0&
{{c}}_{11}&
{{c}}_{{2}1}\\
{{f}}_{1}&
0&
{{c}}_{1{2}}&
{{c}}_{{2}{2}}&
0\\
0&
{{f}}_{1}&
0&
{{c}}_{1{2}}&
{{c}}_{{2}{2}}\\
\end{smallmatrix}
\right)}
\BB^5
\xleftarrow{}
\cdots
\end{equation}
\\
$$
	\cdots
	\xleftarrow{}
	\BB^{2j-1}
	\xleftarrow{\varphi_j}
	\BB^{2j}
	\xleftarrow{\psi_j}
	\BB^{2j+1}
	\xleftarrow{}
	\cdots
$$
\noindent
\renewcommand{\arraystretch}{0.5}
\[
{\scriptsize \scriptstyle
	\varphi_j =
	\left(
	\begin{array}{*3{r@{\,}}r|*3{r@{\,}}r}
	-c_{12} & -c_{22}&&& c_{11} & c_{21}&& \\
	 &\ddots & \ddots && & \ddots & \ddots & \\
	& & -c_{12} & -c_{22}&& & c_{11} & c_{21}  \\
	&&&&&&\\\hline
	&&&&&&\\
	f_1 &&&& f_2 &\\
	& \ddots &&&& \ddots \\
	&& \ddots &&&& \ddots \\
	&&& f_1 &&&& f_2
	\end{array}
	\right), \quad
	\psi_j =
	\left(
	\begin{array}{*2{c@{\,}}c|c@{\,}c@{}c@{\,}c}
	-f_2 & & &c_{11} & c_{21}&& \\
	&\ddots & && \ddots & \ddots & \\
	&& -f_2 & && c_{11} & c_{21}  \\
	&&&&&&\\\hline
	&&&&&&\\
	f_1 &&& c_{12} & c_{22}&\\
	&\ddots&& & \ddots & \ddots & \\
	&&f_1& & & c_{12} & c_{22}
	\end{array}
	\right)}
\]

By choosing a linear map $L:\BB \longrightarrow \KK$ which is non-zero on the 1-dimensional socle on $\BB$,
we obtain a non-degenerate bilinear form:

$$\cdot_L:\BB \times \BB \longrightarrow \KK \hskip 2cm  a\cdot_L b:=L(a\cdot b),$$
where $a\cdot b = ab$ denotes the multiplication of $a$ and $b$ in the ring $\BB$ (see \cite[Sec.21.2]{E}).
We induce also a non-degenerate $\KK$-bilinear form on (the finite $\KK$-dimensional vector space) $\BB^2$ by

$$
 \begin{pmatrix} a_{1 } \cr
a_{2 } \cr
\end{pmatrix} \cdot_L
\begin{pmatrix} b_{1 } \cr
b_{2 } \cr
\end{pmatrix} := L(a_1\cdot b_1+a_2\cdot b_2)$$

Introducing the homology (cohomology) module of the Koszul complexes over
$\BB$ defined by $f_1$ and $f_2$ (Section \ref{3.3}):
 $$
H_1(\cK_\BB(f_1,f_2)) := \frac{<\begin{pmatrix} f_1 \cr f_2 \cr \end{pmatrix}>^\perp}
{<\begin{pmatrix} -f_2 \cr f_1 \cr \end{pmatrix}>}
\ \ , \ \
H^1(\cK_\BB(f_1,f_2)^*) := \frac{<\begin{pmatrix} -f_2 \cr f_1 \cr \end{pmatrix}>^\perp}
{<\begin{pmatrix} f_1 \cr f_2 \cr \end{pmatrix}>},
$$
where   $<*>$ denotes the $\BB$-module generated by $*$ and $*^\perp$ is the $\cdot_L$-orthogonal subspace to $*$.
 The above bilinear form on $\BB$ induces
non-degenerate bilinear forms on the first (co)homology groups. We have that

 $$
\begin{pmatrix} c_{11} \cr
c_{12} \cr
\end{pmatrix},
\begin{pmatrix}  c_{21}\cr
 c_{22}\cr
\end{pmatrix}
\in
H_1(\cK_\BB(f_1,f_2))
\ \ , \ \
\begin{pmatrix} -c_{12} \cr
c_{11} \cr
\end{pmatrix},
\begin{pmatrix}  -c_{22}\cr
 c_{21}\cr
\end{pmatrix}
\in
H^1(\cK_\BB(f_1,f_2)^*)
$$
are elements in the first Koszul  homology  (cohomology) module of
$f_1,f_2$ over $\BB$.
\vskip 5mm

The hyperhomology of the small Gobelin $\cG^1_\BB$ constructed from the first syzygy
$c_{11}f_1+c_{12}f_2=0$ over $\BB$ is:

\begin{prop} \label{lema2}
The hyperhomology groups of the small Gobelin $\cG^1_\BB$
are

\begin{equation}
\HH_0(\cG^1_\BB) = \frac{\BB}{(f_1,f_2)} \hskip 1cm  , \hskip 1cm
\HH_1(\cG^1_\BB) = \frac{H_1(\cK_\BB(f_1,f_2))}
{<\begin{pmatrix} c_{11} \cr c_{12} \cr \end{pmatrix}>}
\end{equation}
and for $j\geq1$
 \begin{equation}
\HH_{2j}(\cG^1_\BB) = \frac
{<\begin{pmatrix} -f_{2} \cr c_{11} \cr \end{pmatrix},
  \begin{pmatrix}  f_{1} \cr c_{12} \cr \end{pmatrix}>^\perp}
{<\begin{pmatrix} -c_{12} \cr f_{1} \cr \end{pmatrix},
\begin{pmatrix} c_{11} \cr f_{2} \cr \end{pmatrix}>}
\hskip 1cm , \hskip 1cm
\HH_{2j+1}(\cG^1_\BB) = \frac
{<
  \begin{pmatrix}  -c_{12} \cr c_{11} \cr \end{pmatrix}>^\perp_{H_1(\cK_\BB(f_1,f_2))}}
{<
\begin{pmatrix} c_{11} \cr c_{12} \cr \end{pmatrix}>_{H_1(\cK_\BB(f_1,f_2))}}
\label{H2G1}
\end{equation}
where these last 2 groups are equidimensional
as $\KK$-vector spaces.
\end{prop}

The invariants that we will   use to describe the hyperhomology of
$\HH_*(\cG^2_\BB)$ in terms of
$\HH_*(\cG^1_\BB)$
are two    flags of
ideals in the ring $ \BB$:

\begin{equation}
\label{flag1}
0 = L _0 \subset L _1 \subset \cdots \subset L _\infty \subset  F_\infty \subset \cdots \subset F_1 \subset F_0 =  \BB
\end{equation}
and
\begin{equation}
\label{flag2}
0 = L^\prime_0 \subset L^\prime_1 \subset \cdots \subset L^\prime_\infty  \subset F^\prime_\infty \subset \cdots \subset F^\prime_1 \subset F^\prime_0 =  \BB
\end{equation}
where the first flag is defined for ascending $j\geq1$ by:
$$L_j:=(L_{j-1}\begin{pmatrix}  c_{21} \cr c_{22}\end{pmatrix} :
\begin{pmatrix}c_{11} \cr c_{12}\end{pmatrix}
)_{H_1(\cK_\BB(f_1,f_2) )}
 \ \ ,\ \
 F_j :=(F_{j-1}\begin{pmatrix} c_{21} \cr c_{22}\end{pmatrix} :
\begin{pmatrix}c_{11} \cr c_{12}\end{pmatrix}
)_{H_1(\cK_\BB(f_1,f_2))}
  \subset \BB$$
and $L_\infty$ and $F_\infty$ are the smallest and largest ideals in $\BB$    satisfying
$I=(I\begin{pmatrix} c_{21} \cr c_{22}\end{pmatrix}:
\begin{pmatrix}c_{11} \cr c_{12}\end{pmatrix}
)$, respectively. The other flag is defined by inverting
the roles of the 2 syzygies:
$$L^\prime_j:=(L^\prime_{j-1}\begin{pmatrix}  c_{11} \cr c_{12}\end{pmatrix} :
\begin{pmatrix}c_{21} \cr c_{22}\end{pmatrix}
)_{H_1(\cK_\BB(f_1,f_2) )}
 \ \ ,\ \
 F^\prime_j :=(F^\prime_{j-1}\begin{pmatrix} c_{11} \cr c_{12}\end{pmatrix} :
\begin{pmatrix}c_{21} \cr c_{22}\end{pmatrix}
)_{H_1(\cK_\BB(f_1,f_2))}
  \subset \BB.$$

Our main result is:

\begin{Teo} \label{CDos}
Let $\BB$ be a   Gorenstein $\KK$-algebra of dimension $0$,
assume given a matrix identity (\ref{syzygyreduced}),
then   the hyperhomology groups of the small Gobelin
$\cG^2_\BB$ are:
$$\HH_{0} (\cG^2_\BB) =\frac{\BB}{(f_1,f_2)}
\hskip 1cm,\hskip1cm
\HH_{1} (\cG^2_\BB) = \frac{ H_1(\cK_\BB(f_1,f_2))}
{<
\begin{pmatrix}c_{11} \cr c_{12}\end{pmatrix},
\begin{pmatrix}c_{21} \cr c_{22}\end{pmatrix}>_{H_1(\cK_\BB(f_1,f_2))}}
$$
  and for $j\geq1$ we have exact sequences:

\begin{equation}
\label{teo1}
0 \xleftarrow{}  \frac{F_{j-1}}{F_{j-1}\cap F_1^\prime}
\xleftarrow{\partial}
\HH_{2j-2}(\cG^{2}_\BB)
\xleftarrow{\sigma_{2j}^* }
\HH_{2j}(\cG^{2}_\BB)
\xleftarrow{i_{2j}}
\frac{\HH_{2j}(\cG^{1}_\BB)}{\frac{\Ann_\BB( L^\prime_j \cap L_1  )}{\Ann_\BB( L_1  )}
\begin{pmatrix}1 \cr 0\end{pmatrix}}
\xleftarrow{}
0
\end{equation}

and

\begin{equation}
\label{teo2}
0 \xleftarrow{}
\frac{\Ann_\BB( L^\prime_j \cap  L_1  )}{\Ann_\BB( L_1  )}
\xleftarrow{\partial}
\HH_{2j-1}(\cG^{2}_\BB)
\xleftarrow{\sigma_{2j+1}^*}
\HH_{2j+1}(\cG^{2}_\BB)
\xleftarrow{i_{2j+1}}
 \frac{\HH_{2j+1}(\cG^{1}_\BB)
}{\frac{F_j}{F_j \cap F_1^\prime}\begin{pmatrix}c_{21} \cr c_{22}\end{pmatrix}}
\xleftarrow{}
0
\end{equation}
 \end{Teo}

 The invariants coming into the theorem are contained
 in the intersection of the flag $L^\prime_*$ with $L_1$
 and in the induced flag  $F_*$ on $\frac{\BB}{F_1^\prime}$.
 Or interchanging the roles of the 2 syzygies,
in the intersection of the flag $L_*$ with $L^\prime_1$
 and in the induced flag  $F_*^\prime$ on $\frac{\BB}{F_1}$.

These computations are useful for giving formulas for the
topological (GSV) and homological indices of a vector field
with isolated singularities tangent
to  a complete intersection with an isolated singularity (see
\cite{BEG,GSV,G}), where $\varphi$ in equation (\ref{E0}) is the Jacobian Matrix $Df$ of $f$ and
(\ref{E0}) is the tangency condition of the vector field along the complete intersection $f=0$.
In the isolated singularity case, the $\BB$ algebra is the $\KK$-algebra of functions on the zero locus $X=0$
of the vector field, which is an isolated but multiple point. So the algebra
we develop here is inside the zero set of the vector field.

\section{The small Gobelin}

All tensor products will be over $\KK$,
unless otherwise specified.
Let
$F$, $G$ and $H$ be finite dimensional $\KK$-vector spaces of dimensions $N$, $\ell$, and $k$ respectively.
Then the equation (\ref{E0})
gives rise to the anticommutative square
\begin{equation}
\begin{matrix}
\cO &
{\overset{-f}{\longrightarrow}}&
H \otimes\cO\cr
X \downarrow \ \ &&\downarrow
 c\cr
F \otimes \cO&
{\overset{\longrightarrow}{ \varphi}}
&
G \otimes \cO \cr
\end{matrix}
\label{E1}
\end{equation}
Let $\PP^{\ell - 1}$
denote the projective space, ${\rm Proj}(G),$ and
$\cO_{\PP^{\ell-1}}(1)$ the sheaf of hyperplane sections on $\PP^{\ell-1}$.
 Let $s_1,\ldots,s_\ell$ denote a basis of its global sections, $s:=(s_1,\ldots,s_\ell)$,
$\widetilde \cO :=\cO \otimes\cO_{\PP^{\ell-1}}$ and
$\widetilde \cO(m) :=\cO \otimes\cO_{\PP^{\ell-1}}(1)^{\otimes m}$.
We tensor the square (\ref{E1}) with the sheaf $\cO_{\PP^{\ell-1}}$
and continue at the right bottom of the square
with the tensor product of the natural morphism
$$s\cdot\  :G\otimes \cO_{\PP^{\ell-1}}
 \longrightarrow \cO_{\PP^{\ell-1}}(1)$$
with $\cO$ to obtain the following anticommutative square
of $\widetilde\cO$-sheaves on $\PP^{\ell-1}$:
$$
\begin{matrix}
\widetilde\cO &
{\overset{-f}{\longrightarrow}}&
H \otimes\widetilde\cO\cr
X \downarrow \ \ &&\downarrow
s \cdot c\cr
F \otimes \widetilde\cO&
{\overset{\longrightarrow}{s \cdot \varphi}}
&
\widetilde\cO(1) \cr
\end{matrix}
$$

Since going around the square gives a $1  \times 1$ matrix,
we can transpose the upper part of the square and
obtain the anticommutative square

\begin{equation}
\begin{matrix}
\widetilde\cO &
\xrightarrow{(s \cdot c)^t}&
H^* \otimes\widetilde\cO(1)\cr
X \downarrow \ \ &&\downarrow
-f^t\cr
F \otimes \widetilde\cO&
{\overset{\longrightarrow}{s \cdot \varphi}}
&
\widetilde\cO(1) \cr
\end{matrix}.
\label{(0.2)}
\end{equation}
It is explained in \cite{BEG} how from this data one constructs
a double complex, called the Gobelin.

Over  the ring $\BB :=\cO/(X_1,\ldots,X_N)$ the identity
(\ref{E0})   reduces to

\begin{equation}
cf = \begin{pmatrix}
c_{11} & \ldots & c_{1k}\cr
\vdots & \ddots & \vdots\cr
c_{\ell 1} & \ldots & c_{\ell k}\cr
\end{pmatrix}
\begin{pmatrix}f_1\cr
\vdots\cr
f_k\cr
\end{pmatrix}
= 0.
\label{IdentMatrix1}
\end{equation}

Using the notation
$\widetilde \BB_{\PP^{\ell-1}}:= \BB \otimes \cO_{\PP^{\ell-1}}$
and $\widetilde \BB_{{\PP^{\ell-1}}}(1):=
\BB \otimes \cO_{\PP^{\ell-1}}(1)$,
(\ref{(0.2)})  gives rise to a syzygy

\begin{equation*}
\widetilde\BB_{\PP^{\ell-1}}
\xrightarrow{ (s\cdot c)^t}
H^* \otimes\widetilde\BB_{\PP^{\ell-1}}(1)
\xrightarrow{-f^t}
\widetilde\BB_{\PP^{\ell-1}}(1).
\end{equation*}

Considering the Koszul complexes
formed with each term of the syzygy, tensoring with $\cO_{\PP^{\ell-1}}(d)$,
taking global sections and dualizing we weave these
Buchsbaum-Eisenbud
complexes (see \cite[Pag. 589]{E}) to
form a double complex,  called the
small Gobelin in \cite{BEG}, which we denote by
$$\cG_\BB^\ell:=\{\cG_{\BB,i,j}^\ell:=
D_iG^*\otimes\Lambda^{\ell+i-j}H\otimes\BB\}$$
where $D_iG^*:=H^0(\PP^{\ell-1},\cO_{\PP^{\ell-1}}(i))^*$ is the
homogeneous component of the divided power algebra
of $\KK[s_1,\ldots,s_\ell]$ of degree $i$,
and the connecting maps are constructed using
$f$ for the vertical strands and
$c$ for the horizontal ones.

\begin{figure}[h]
{\tiny
$$
\hskip -15mm
\begin{array}{*{18}{c@{\,}}}

&
&&
&&
&&
&&
& \downarrow &
& \downarrow &
\\
\\

&
&&
&&
&&
&0 & \longleftarrow
& D_4G^* \otimes \BB & \longleftarrow
&
D_5G^*\otimes H\otimes  \BB
\\ \\

&
&&
&&
&&
& \downarrow &
& \downarrow &
& \downarrow &

\\ \\

&&
&&&& 0 & \longleftarrow
&D_3G^* \otimes \BB & \longleftarrow
&
D_4G^*\otimes H\otimes  \BB
 & \longleftarrow
&
D_5G^*\otimes\Lambda^2H \otimes \BB
&

\\ \\

&
&&
&&
& \downarrow &
& \downarrow &
& \downarrow &
& \downarrow &

\\ \\

&&&&
 0 & \longleftarrow
&D_2G^*\otimes\BB & \longleftarrow
&
D_3G^*\otimes H  \otimes \BB
 & \longleftarrow
&
D_4G^*\otimes\Lambda^2H \otimes\BB
 & \longleftarrow
& 0

\\ \\

 &&&
& \downarrow &
& \downarrow &
& \downarrow &
& \downarrow &

\\ \\

&&0          & \longleftarrow
& G^* \otimes \BB & \longleftarrow
&
D_2G^*\otimes H \otimes \BB
 & \longleftarrow
&
D_3G^*\otimes\Lambda^2H \otimes \BB
 & \longleftarrow
& 0

\\ \\

&&\downarrow  &
& \downarrow &
& \downarrow &
& \downarrow &

\\ \\

0          & \longleftarrow &
 \BB & \longleftarrow
&
G^*\otimes  H\otimes \BB
 & \longleftarrow
&
D_2G^*\otimes\Lambda^2H \otimes \BB
 & \longleftarrow
&0

\\ \\

&& \downarrow &
& \downarrow &
& \downarrow &

\\ \\

0          & \longleftarrow &
 H\otimes \BB & \longleftarrow
&

G^*\otimes\Lambda^2H\otimes \BB

& \longleftarrow
& 0

\\ \\

&
& \downarrow &
& \downarrow &

\\ \\

0          & \longleftarrow &
 \Lambda^2H\otimes \BB & \longleftarrow
& 0

\\ \\

&
& \downarrow
\\ \\

&
& 0
\\ \\
\end{array}
$$}
\caption{The lower left hand part of the small Gobelin
$\cG_\BB^2$ for   $k=2$, $l=2$
beginning at $(0,0)$.} \label{Gobelin22}
\end{figure}

\section{Morphisms between small Gobelins}

Let $r< \ell$. We will modify the identity (\ref{IdentMatrix1}),
considering only the first $r$ rows, that is:

\begin{equation}\label{IdentMatrix2}
\begin{pmatrix}
c_{11} & \ldots & c_{1k}\cr
\vdots & \ddots & \vdots\cr
c_{r 1} & \ldots & c_{r k}\cr
\end{pmatrix}
\begin{pmatrix}f_1\cr
\vdots\cr
f_k\cr
\end{pmatrix}
= 0.
\end{equation}
Let $G_r \subset G$ be the vector subspace obtained by setting
$s_{r+1}=\ldots=s_r=0$ and denote by
   $\cG_\BB^{r}$
 the small Gobelin constructed from the
syzygy associated to the equation (\ref{IdentMatrix2})

\begin{equation*}
\widetilde{\BB}_{\PP^{r-1}}
\xrightarrow{ (s^\prime\cdot c)^t}
H^* \otimes\widetilde\BB_{\PP^{r-1}}(1)
\xrightarrow{-f^t}
\widetilde{\BB}_{\PP^{r-1}}(1).
\end{equation*}
with $\PP^{r-1} :=\hbox{Proj}(G_r) \subset \PP^{\ell-1}$ and $s^\prime:=(s_1,\ldots,s_r)$.
Let $D_iG^*_{r}  :=H^0(\PP^{r-1},\cO_{\PP^{r-1}}(i))^*$ be the
homogeneous component of the divided power algebra
of $\KK[s_1,\ldots,s_r]$ of degree $i$.

Consider the maps induced by multiplication by $s_r$:

$$
\sigma :H^0(\PP^{r-1},\cO_{\PP^{r-1}}(i))
\rightarrow
H^0(\PP^{r-1},\cO_{\PP^{r-1}}(i+1))
\hskip 1cm
\quad s^{\alpha }\rightarrow s^{\alpha} s_{r}.
$$

\noindent
Note that
 $$
H^0(\PP^{r-2},\cO_{\PP^{r-2}}(i+1))
=\frac{H^0(\PP^{r-1},\cO_{\PP^{r-1}}(i+1))
}{ \sigma (H^0(\PP^{r-1},\cO_{\PP^{r-1}}(i)))}
$$
These morphisms
induces by duality the morphisms

$$
\sigma^* : \cG^{r}_{\BB,i,j}=D_iG^*_r \otimes
\wedge^{r+i-j} H \otimes  \BB
\rightarrow
\cG^{r}_{\BB,i-1,j-1}=D_{i-1}G^*_r \otimes
\wedge^{r+i-j} H \otimes \BB
$$
\noindent
$$
a \otimes  \eta \otimes b  \rightarrow
\sigma ^* ( a ) \otimes \eta  \otimes b.
$$
\noindent which consists of contracting with the last variable $s_r$.
 Let $\cG^{r}_\BB (-1,-1)$ denote the shift of $\cG^{r}_\BB$ by $(-1,-1)$, i.e.
$$
\cG^{r}_{\BB} (-1,-1)_{i,j}=\cG^{r}_{\BB,i-1,j-1}.
$$
\noindent Hence $\sigma^*$ defines
 a morphism between the double complexes
$\cG^{r }_\BB$ and $\cG^{r }_\BB (-1,-1)$.
Note that in figure \ref{Gobelin22}, $\sigma^*$ corresponds to
 antidiagonal arrows descending vertically one and moving horizontally
to the left by one.\\
\\

Let $\iota:\cG^{r-1 }_\BB\rightarrow \cG^{r }_\BB$
be the inclusion defined in each term of the double complexes
by
$$
\iota : \cG^{r-1}_{\BB,i,j}=
D_{i}G^*_{r-1} \otimes \wedge^{r-1+i-j} H   \otimes\BB \rightarrow
\cG^{r }_{\BB,i,j+1}=
D_{i}G^*_r \otimes \wedge^{r+i-j-1} H \otimes \BB.
$$

Directly from above we obtain
\begin{Teo}\label{ExactaGobelinos}
The following is a short exact sequence of double complexes of $\BB$-modules
$$
0\FlechaDer{}\cG^{r-1}_\BB\FlechaDer{\iota}
 \cG^{r}_\BB\FlechaDer{\sigma^*}\cG^{r}_\BB (-1,-1)\FlechaDer{} 0.
$$
\noindent
giving rise to a long exact sequence of hyperhomology groups

$$
\ldots \FlechaDer{} \HH_j(\cG^{r-1}_\BB )\FlechaDer{\iota }
\HH_j(\cG^{r}_\BB )\FlechaDer{\sigma^*}
\HH_{j-2}(\cG^{r}_\BB )\FlechaDer{\delta}
\HH_{j-1}(\cG^{r-1}_\BB ) \FlechaDer{} \dots.
$$
\end{Teo}

This long exact sequence
relates the hyperhomology groups  $\HH_j(\cG^{r}_\BB )$
of the small Gobelin $\cG^r_\BB$ via
$\sigma^*$ with
$\HH_{j-2} (\cG^{r}_\BB )$
and we find the kernels
and cokernels in the homology in the
small Gobelin $\cG^{r-1}_\BB$ with $(r-1)$-syzygies between $f_1,\ldots, f_k$.
In particular
when $r=1$, the kernels and cokernels will be $0$,
giving the isomorphism
between even and odd homology groups for positive degrees, as in \cite{G}
for the hypersurface case.

\section{Some auxiliary algebra}
\label{Sec3}

To develop the pattern that follows from Theorem \ref{ExactaGobelinos},
we will analyse the case where $k=2$ and $\ell=1,2$.
We will assume from now on further that $\cO$ is a local
Gorenstein $\KK$-algebra of dimension $N$, with $\KK$ as residue field
and that $X_1,\ldots,X_N$ is a regular sequence.
Hence we have that $\BB$
is a local Gorenstein $\KK$-algebra of dimension $0$
which is a finite dimensional vector space over $\KK$,
say of dimension $\mu$. Multiplication in $\BB$ will be denote by $a\cdot b$, or simply $ab$.

\subsection{The ideals associated to $f_1$ and $f_2$.}

Let $f_1,f_2 \in\BB$ be non-units. Denote by $\nu_1$, $\nu_2$ and
$\nu$ the codimension of the ideals $(f_1), (f_2) $ and $(f_1,f_2)$ in $\BB$.
It follows from the short exact sequence

$$0 \longrightarrow (f_1) \cap (f_2)
 {\stackrel{(id,-id)}\rightarrow} (f_1) \oplus (f_2)
{\stackrel{+}\rightarrow}(f_1,f_2)\longrightarrow 0
$$
that the codimension of $(f_1) \cap (f_2) $ is
$  \nu_1+\nu_2- \nu $ and from the non-degenerate duality
that $\Ann_\BB(f_1) \cap \Ann_\BB(f_2) = \Ann_\BB(f_1,f_2)$
has dimension $\nu$, $\Ann_\BB(f_i)$  has dimension $\nu_i$ and
$\Ann_\BB((f_1)\cap(f_2))$ has  dimension $\nu_1+\nu_2-\nu$:

$$
\begin{matrix}
 &\BB &  \cr
 & \cup&  \cr
 &(f_1,f_2)  & \cr
\cup & &\cup  \cr
(f_1)  & & (f_2) \cr
\cup & &\cup  \cr
 &(f_1) \cap (f_2)  & \cr
 & \cup&  \cr
 &0 &  \cr
\end{matrix}
\hskip 2cm
\begin{matrix}
 &\BB &  \cr
 & \cup&  \cr
 \Ann_\BB((f_1)\cap(f_2))& =& < \Ann_\BB(f_1),\Ann_\BB(f_2)>  \cr
\cup & &\cup  \cr
\Ann_\BB(f_1)  & & \Ann_\BB(f_2) \cr
\cup & &\cup  \cr
 \Ann_\BB(f_1,f_2)& =& \Ann_\BB(f_1) \cap \Ann_\BB(f_2)   \cr
 & \cup&  \cr
 &0 &  \cr
\end{matrix}
$$

From the exact sequence
$$0 \longrightarrow \Ann_\BB(f_1)
 \rightarrow \BB
{\stackrel{f_1}\rightarrow}(f_1)\longrightarrow 0
$$
we deduce the sequence
$$0 \longrightarrow \Ann_\BB(f_1)
 \rightarrow (f_2:f_1)
{\stackrel{f_1}\rightarrow}(f_1) \cap (f_2)\longrightarrow 0
$$
so that the dimension of $(f_2:f_1)$ is $\mu-\nu_2+\nu$,
and similarly the dimension of $(f_1:f_2)$ is $\mu-\nu_1+\nu$.
The modules
$$\frac{(f_2:f_1)}{(f_2)}, \hskip 1cm\hskip1cm \frac{(f_1:f_2)}{(f_1)}$$
have dimension $\nu$. From the exact sequence
$$0 \longrightarrow \Ann_\BB(f_1) \cap \Ann_\BB(f_2)
 \rightarrow \Ann_\BB(f_2)
{\stackrel{f_1}\rightarrow}  f_1\Ann_\BB(f_2)\longrightarrow 0
$$
we see that the dimension of $f_1\Ann_\BB(f_2))$ is $\nu_2-\nu$ and similarly
the dimension of $f_2\Ann_\BB(f_1))$ is $\nu_1-\nu$.
Hence the modules
$$\frac{\Ann_\BB(f_2)}{f_1\Ann_\BB(f_2)}, \hskip 1cm\hskip1cm
\frac{\Ann_\BB(f_1)}{f_2\Ann_\BB(f_1)}
$$
have dimension $\nu$.

\subsection{Non-degenerate bilinear forms}

Let $L:\BB \longrightarrow \KK$ be a trace map, i.e. a
functional which is non trivial on the socle of $ \BB$
(see \cite{BH,HK}).
The  bilinear form on $\BB$

\begin{equation}\label{d1}
\cdot_{L}: \BB \oplus \BB {\stackrel{\cdot}\rightarrow} \BB
{\stackrel{L }\rightarrow} \KK,\hskip 2cm a \cdot_L b = L(a \cdot b),
\end{equation}
is non-degenerate, having the property that if $I$ is an ideal in $\BB$
then its orthogonal $I^\perp$
is $\Ann_\BB(I)$, which is independent of the chosen
trace $L$.

Introduce  on $\BB^r$ the direct sum
non-degenerate symmetric bilinear forms:

\begin{equation}
 \label{d2a}
\cdot:\BB^r \oplus \BB^r \longrightarrow \BB, \hskip 1cm
 \begin{pmatrix} a_1 \cr \vdots \cr  a_r   \end{pmatrix}
\cdot \begin{pmatrix} b_1 \cr \vdots\cr  b_r   \end{pmatrix} :=
\begin{pmatrix} a_1 \cr \vdots\cr  a_r   \end{pmatrix}^t
\begin{pmatrix} b_1 \cr \vdots\cr  b_r   \end{pmatrix}=
\sum_{i=1}^r a_ib_i,
 \end{equation}

 \begin{equation}
 \label{d2}
\cdot_L:\BB^r \oplus \BB^r \longrightarrow \KK, \hskip 1cm
\begin{pmatrix} a_1 \cr \vdots \cr  a_r   \end{pmatrix}
\cdot_L \begin{pmatrix} b_1 \cr \vdots\cr  b_r   \end{pmatrix} :=
L(\begin{pmatrix} a_1 \cr \vdots\cr  a_r   \end{pmatrix}^t
\begin{pmatrix} b_1 \cr \vdots\cr  b_r   \end{pmatrix})=
L\left(\sum_{i=1}^r a_ib_i\right),
 \end{equation}

Given a submodule $M \subset \BB^r$, we will denote by
$M^\perp \subset \BB^r$ its orthogonal. It has complementary dimension, due to the
non-degeneracy of the bilinear form (\ref{d2}), and it is also a submodule. If $m_1,\ldots,m_s$
are generators of a module $M$ we will denote by $<m_1,\ldots,m_r>$ the module generated by them
and by $<m_1,\ldots,m_r>^\perp$ its orthogonal. If $M_1,M_2$ are submodules of
$\BB^r$, then $M_1^\perp \cap M_2^\perp = <M_1 \cup M_2>^\perp$.

Note that the involution
\begin{equation}
\kappa:\BB^2 \longrightarrow \BB^2,
\hskip 1cm\hskip 1cm \kappa(a,b) := (-b,a),
\label{kappa}
\end{equation}
is an $\cdot_L$-automorphism:
$$ \kappa
\begin{pmatrix} a  \cr  b   \end{pmatrix}
\cdot_L \kappa\begin{pmatrix} c \cr  d   \end{pmatrix}
  =
\begin{pmatrix} -b \cr  a   \end{pmatrix}
\cdot_L \begin{pmatrix} -d \cr   c  \end{pmatrix}
=L(bd+ac)=
\begin{pmatrix} a \cr  b   \end{pmatrix}
\cdot_L \begin{pmatrix} c \cr   d   \end{pmatrix}
  .$$
\vskip 3mm
\subsection{The Koszul complexes  $\cK_\BB(f_1,f_2)$ and $\cK_\BB(f_1,f_2)^*$} \label{3.3}

Denote by $\cK_\BB(f_1,f_2)$   the Koszul complex built
with $f_1,f_2$ over $\BB$:

\begin{equation}\label{d3}
0 \longleftarrow \BB {\stackrel{ ( f_1, f_2)}\longleftarrow}
 \BB^2 {\stackrel{\begin{pmatrix} -f_2 \cr f_1 \cr \end{pmatrix}}\longleftarrow}
 \BB \longleftarrow 0
\end{equation}
with homology groups
\begin{equation} \label{Koszul}
H_0(\cK_\BB(f_1,f_2)) = \frac{\BB}{(f_1,f_2)}, 
H_1(\cK_\BB(f_1,f_2)) = \frac{<\begin{pmatrix} f_1 \cr f_2 \cr \end{pmatrix}>^\perp}
{<\begin{pmatrix} -f_2 \cr f_1 \cr \end{pmatrix}>},
\hbox{ and } 
H_2(\cK_\BB(f_1,f_2))= \Ann_\BB(f_1,f_2).
\end{equation}
The dimension of $\frac{\BB}{(f_1,f_2)}$ is $\nu$,
  $H_2(\cK_\BB(f_1,f_2))$ also has dimension $\nu$
since it is orthogonal to $(f_1,f_2)$ and
$H_1(\cK_\BB(f_1,f_2))$ has dimension $2\nu$,
since the alternating sum of the dimensions of the modules in (\ref{d3}) is 0,
and hence also the Euler characteristic of its homology groups.

From the exact sequence
$$0 \longrightarrow \Ann_\BB(f_1) \cap \Ann_\BB(f_2)
 \rightarrow \BB
{\stackrel{
\begin{pmatrix} f_1 \cr f_2 \cr \end{pmatrix}
}\rightarrow}  <\begin{pmatrix} f_1 \cr f_2 \cr \end{pmatrix}> \longrightarrow 0
$$
we also see that the dimension of
$<\begin{pmatrix} f_1 \cr f_2 \cr \end{pmatrix}>$
 is $\mu-\nu$,
(similarly of $<\begin{pmatrix} -f_2 \cr f_1 \cr \end{pmatrix}> $) and
by orthogonality in $\BB^2$, the dimension of
$<\begin{pmatrix} f_1 \cr f_2 \cr \end{pmatrix}>^\perp$
and of $<\begin{pmatrix} -f_2 \cr f_1 \cr \end{pmatrix}>^\perp $ is $\mu+\nu$.
This gives another proof that
 $H_1(\cK_\BB(f_1,f_2)) $ has dimension $2\nu$.

\begin{lema} \label{lema1}
The inclusions $i$ and projections $\pi$ of $\BB^2$ to the factors
induces the exact sequences:

$$0 \longrightarrow \frac{\Ann_\BB(f_2)}{f_1 \Ann_\BB(f_2)}
 {\stackrel{
i_2
}\rightarrow}   H_1(\cK_\BB(f_1,f_2))
{\stackrel{\pi_1}
\rightarrow}  \frac{(f_2:f_1)}{(f_2)} \longrightarrow 0
$$

$$0 \longrightarrow \frac{\Ann_\BB(f_1)}{f_2 \Ann_\BB(f_1)}
 {\stackrel{
i_1
}\rightarrow}   H_1(\cK_\BB(f_1,f_2))
{\stackrel{\pi_2}
\rightarrow}  \frac{(f_1:f_2)}{(f_1)} \longrightarrow 0
$$
\end{lema}
\begin{proof}
We prove the second one. The projection $\pi_2$ induces a map
$ H_1(\cK_\BB(f_1,f_2))
{\stackrel{\pi_2}
\rightarrow}  \frac{\BB}{(f_1)}$
due to the boundaries $c\begin{pmatrix} -f_2 \cr f_1 \end{pmatrix}$.
 Now $\begin{pmatrix} a \cr b \end{pmatrix} \in
H_1(\cK_\BB(f_1,f_2))$ means $af_1+bf_2=0$,
hence $bf_2=-af_1$, so $b\in (f_1:f_2)$.
And conversely, given $b\in (f_1:f_2)$ there is an $a\in \BB$
with $bf_2=-af_1$, and we may form the class $(a,b)^t  \in
H_1(\cK_\BB(f_1,f_2))$.

An element in $\Ker(\pi_2)
\subset H_1(\cK_\BB(f_1,f_2))$ may be written as
 $\begin{pmatrix} a \cr 0 \end{pmatrix}
+b\begin{pmatrix} -f_2 \cr f_1 \end{pmatrix}
\in \begin{pmatrix} f_1 \cr f_2 \end{pmatrix}^\perp$
that is, $af_1=0$, and hence $a \in \Ann_\BB(f_1)$.
The element $\begin{pmatrix} a \cr 0 \end{pmatrix}$
is $0\in
H_1(\cK_\BB(f_1,f_2))$ if $\begin{pmatrix} a \cr 0 \end{pmatrix}=c
\begin{pmatrix} -f_2 \cr f_1 \end{pmatrix}$, hence
$a=-cf_2$ with $cf_1=0$; that is, $a\in f_2\Ann_\BB(f_1)$.
\end{proof}

  The algorithm to construct an element in
$H_1(\cK_\BB(f_1,f_2))$ is then to begin with an
element $b \in (f_1:f_2) \subset \BB$, so
it satisfies a relation $bf_2=-af_1$ for some $a\in \BB$,
and then form $\begin{pmatrix} a \cr b \end{pmatrix}
\in H_1(\cK_\BB(f_1,f_2))$. The $a$ is unique up to adding
an element in $\Ann_\BB(f_1)$. If $b=cf_1$ then
$\begin{pmatrix} a \cr b \end{pmatrix}=c
\begin{pmatrix} -f_2 \cr f_1 \end{pmatrix}$ represents
the $0$ element.

The dual Koszul complex $\cK_\BB(f_1,f_2)^*$ is

$$
0 \longrightarrow \BB {\stackrel{
\begin{pmatrix} f_1 \cr f_2 \cr \end{pmatrix}
  }\longrightarrow}
 \BB^2 {\stackrel{( -f_2, f_1)}\longrightarrow}
 \BB \longrightarrow 0
$$
with cohomology groups
$$
H^0(\cK_\BB(f_1,f_2)^*) = \Ann_{\BB}{(f_1,f_2)}, \
H^1(\cK_\BB(f_1,f_2)^*) = \frac{<\begin{pmatrix} -f_2 \cr f_1 \cr \end{pmatrix}>^\perp}
{<\begin{pmatrix} f_1 \cr f_2 \cr \end{pmatrix}>}, \hbox{ and }
H^2(\cK_\BB(f_1,f_2)^*)= \frac{\BB}{(f_1,f_2)}.
$$

The bilinear form $\cdot_L$ (\ref{d2a}) in $\BB^2$ induces  non-degenerate
bilinear forms in $H_1(\cK_\BB(f_1,f_2))$ and in  $H^1(\cK_\BB(f_1,f_2)^*)$, and the $\BB$-isomorphism
$\kappa$ in (\ref{kappa}) induces a $\BB$-module isomorphism
$$\kappa^*: H_1(\cK_\BB(f_1,f_2)) \longrightarrow
H^1(\cK_\BB(f_1,f_2)^*)$$
which preserves the bilinear forms.

\section{The hyperhomology of the Gobelin $\cG^1_\BB$.}

\subsection{The hyperhomology of $\cG^1_\BB$.}

The small Gobelin $\cG_\BB^{1}$ constructed from
\\
\begin{equation}
\begin{pmatrix}
c_{11} &  c_{12}\cr
\end{pmatrix}
\begin{pmatrix}f_1\cr
f_2\cr
\end{pmatrix}
=
0
\label{BMatrixIdentity}
\end{equation}
is given by Figure 1 with $G=G_1$.
The total complex $\cG^{1}_\BB$ has the form
\begin{equation}
 0  \FlechaIzq{} \BB  \FlechaIzq{(f_{1},f_{2})}\BB^2   \FlechaIzq{C_{\psi }}\BB^2   \FlechaIzq{C_{\varphi }}\BB^2   \FlechaIzq{C_{\psi }}
\BB^2   \FlechaIzq{C_{\varphi }}\cdots
\label{HyperHomologySequenceG1}
\end{equation}
\noindent where

$$
C_{\psi }=
\begin{pmatrix}
-f_{2} & c_{11}\cr
f_{1}  & c_{12}\cr
\end{pmatrix}
,\quad
C_{\varphi}=
\begin{pmatrix}
-c_{12} & c_{11}\cr
f_{1}  & f_{2}\cr
\end{pmatrix}.
$$

We have $\begin{pmatrix} c_{11} \cr c_{12} \cr \end{pmatrix} \in
<\begin{pmatrix} f_1 \cr f_2 \cr \end{pmatrix}>^\perp$, i.e.
it satisfies a relation in $\BB$ of the form $c_{11}f_1+c_{12}f_2=0$.

\begin{proof}[Proof of Proposition \ref{lema2}]
The computation of $\HH_0(\cG^1_\BB)$ is direct from the complex
(\ref{HyperHomologySequenceG1}), as also
$$\HH_1(\cG^1_\BB)=
\frac{<\begin{pmatrix}  f_{1} \cr f_{2} \cr \end{pmatrix}>^\perp}{
<\begin{pmatrix}  -f_{2} \cr f_{1} \cr \end{pmatrix},
\begin{pmatrix}  c_{11} \cr c_{12} \cr \end{pmatrix}>}
=\frac
{H_1(\cK_\BB(f_1,f_2))}
{<\begin{pmatrix}  c_{11} \cr c_{12} \cr \end{pmatrix}>}
$$

The 2-periodicity of $\HH_j(\cG^1_\BB)$ follows from the
2-periodicity of
(\ref{HyperHomologySequenceG1}).
One computes directly
$$\frac{\Ker(C_\psi)}{\Imag(C_\varphi)}, \hskip 1cm\hbox{and}\hskip 1cm
\frac{\Ker(C_\varphi)}{\Imag(C_\psi)}$$
to obtain (\ref{H2G1}).

We also have for $i,j\geq 1$ \\
\\
$\dim (\HH_{2j}(\cG^{1}_\BB))=\dim (\Ker( C_\psi ))-\dim (\Imag(C_\varphi)) =\dim (\BB^2) -\dim (\Imag(C_\psi))-\dim (\Imag(C_\varphi))$\\
\- $ $ \- $ $ \- $ $ \- $ $ \- $ $ \- $ $ \- $ $ \- $ $ \- $ $ \- $ $
$=\dim (\Ker( C_\varphi ))-\dim (\Imag(C_\psi))= \dim (\HH_{2i+1}(\cG^{1}_\BB))$.\\
\\
Thus for $j\geq 2$   all the homology groups $\HH_{j}(\cG^{1}_\BB)$ are  equidimensional
as $\KK$-vector spaces.\\
\end{proof}

\begin{lema} \label{lema4homology}
For $j\geq1$
the projection map to the second factor
\begin{equation}\label{pi1homologia}
\pi_2: \HH_{2j}(\cG^{1}_\BB) \longrightarrow \frac{\BB}{(f_1,f_2)},
\hskip 1cm \pi_2 (a,b) = b,
\end{equation}
induce an exact sequence:

\begin{equation} \label{H2jhomologia}
0 \longrightarrow \frac{\Ann_\BB(f_1,f_2)}
{\begin{pmatrix} -c_{12} \cr c_{11} \cr \end{pmatrix}
\cdot
\frac{H_1(\cK_\BB(f_1,f_2))}{
<\begin{pmatrix} c_{11} \cr c_{12} \cr \end{pmatrix}>}}
\FlechaDer{i_1}
\HH_{2j}(\cG^{1}_\BB)
\FlechaDer{\pi_2}
\frac{(0:
\begin{pmatrix} c_{11} \cr c_{12} \cr \end{pmatrix})_{H_1(\cK_\BB(f_1,f_2))}}{(f_1,f_2)}
\longrightarrow
0
\end{equation}
\end{lema}

\begin{proof}
$\begin{pmatrix} a \cr b \cr \end{pmatrix} \in \HH_{2j}(\cG^{1}_\BB)$ satisfies $-af_{2}+bc_{11}= 0$ and
$af_{1}+bc_{12}=0$, or in matrix notation
$$
a\begin{pmatrix} -f_{2} \cr f_{1} \cr \end{pmatrix}
+b\begin{pmatrix} c_{11} \cr c_{12} \cr \end{pmatrix}=0\ \ \
\hbox{so that} \ \ \  b \in (\begin{pmatrix} -f_{2} \cr f_{1} \cr \end{pmatrix}:
\begin{pmatrix}  c_{11} \cr c_{12} \cr \end{pmatrix}),$$ and if its cohomology class is 0
then
$\begin{pmatrix} a \cr b \cr \end{pmatrix} \in <\begin{pmatrix} -c_{12} \cr f_{1} \cr \end{pmatrix},
\begin{pmatrix} c_{11} \cr f_{2} \cr \end{pmatrix}>$ so that
$b \in (f_1,f_2)$. This shows that the map $\pi_2$ is well defined and surjective.

The elements   in the kernel of $\pi_2$  may be represented in the form $(a^\prime,\alpha_1f_1+\alpha_2f_2)^t$,
so that it has a homologous element of the form

$$\begin{pmatrix} a \cr 0 \cr \end{pmatrix}=
\begin{pmatrix}a^\prime \cr \alpha_1f_{1} + \alpha_2f_{2} \cr  \end{pmatrix}-\alpha_1
\begin{pmatrix} -c_{12} \cr f_{1} \cr \end{pmatrix}
-\alpha_2
\begin{pmatrix} c_{11} \cr f_{2} \cr \end{pmatrix}.$$
The element $(a,0)^t$ belongs to $\HH_{2j}(\cG^{1}_\BB)$ if and only if $a\in \Ann_\BB(f_1,f_2)$.
It is the $0$ element if and only if

$$
\begin{pmatrix} a \cr 0 \cr \end{pmatrix}=\alpha
\begin{pmatrix} -c_{12} \cr f_{1} \cr \end{pmatrix}+\beta
\begin{pmatrix} c_{11} \cr f_{2} \cr \end{pmatrix}=
\begin{pmatrix}
(\alpha,\beta)^t\cdot (-c_{12},c_{11})^t  \cr
(\alpha,\beta)^t\cdot (f_1,f_2)^t    \cr \end{pmatrix}
 $$

Hence $(\alpha,\beta)^t\in <(f_1,f_2)^t>^\perp$. Consider the map
$$\cdot\begin{pmatrix} -c_{12} \cr c_{11} \cr \end{pmatrix} :
<\begin{pmatrix} f_1 \cr f_2 \cr \end{pmatrix}>^\perp \longrightarrow \BB$$
which vanishes on
$
<\begin{pmatrix} -f_2 \cr f_1 \cr \end{pmatrix},
\begin{pmatrix}  c_{11} \cr c_{12} \cr \end{pmatrix}>$.
Its image lies in $\Ann_\BB(f_1,f_2)$ and consists of those
elements which represent  0 in $\Ker(i_1)$.
\end{proof}

\subsection{The hypercohomology of $\cG^{1*}_\BB$.}

 Let $\cG^{1*}_\BB$ be the double complex $\Hom_\BB(\cG^1_\BB,\BB)$.
The total complex $\tot(\cG^{1*}_\BB)$ is obtained by applying the functor
 $\Hom_\BB(*,\BB)$ to (\ref{HyperHomologySequenceG1}):

$$
 0  \FlechaDer{} \BB  \FlechaDer{(f_{1},f_{2})^t}\BB^2
 \FlechaDer{C_{\psi }^t}\BB^2   \FlechaDer{C_{\varphi }^t}\BB^2   \FlechaDer{C_{\psi }^t}
\BB^2   \FlechaDer{C_{\varphi }^t}\cdots
$$

\begin{lema} \label{lema3}
The hypercohomology groups of $\cG^{1*}_\BB$
are

$$
\HH^0(\cG^{1*}_\BB) = \Ann_{\BB}{(f_1,f_2)}, \hskip 1cm \hskip 1cm
\HH^1(\cG^{1*}_\BB) =
 {<\begin{pmatrix} c_{11} \cr c_{12} \cr \end{pmatrix}>^\perp_{H^1(\cK_\BB(f_1,f_2))} }
$$
and for $j\geq1$
$$
\HH^{2j}(\cG^{1*}_\BB) = \frac
{<\begin{pmatrix} -c_{12} \cr f_{1} \cr \end{pmatrix},
  \begin{pmatrix}  c_{11} \cr f_{2} \cr \end{pmatrix}>^\perp}
{<\begin{pmatrix} -f_{2} \cr c_{11} \cr \end{pmatrix},
\begin{pmatrix} f_{1} \cr c_{12} \cr \end{pmatrix}>},
\hskip 1cm \hskip 1cm
\HH^{2j+1}(\cG^{1*}_\BB) = \frac
{<\begin{pmatrix} c_{11} \cr c_{12} \cr \end{pmatrix}>^\perp_{H^1(\cK_\BB(f_1,f_2))} }
{<
\begin{pmatrix} -c_{12} \cr c_{11} \cr \end{pmatrix}>_{H^1(\cK_\BB(f_1,f_2))}}
$$
where these last 2 groups are equidimensional
as $\KK$-vector spaces.
\end{lema}

\begin{proof} Similar as the proof of Proposition \ref{lema2}.
\end{proof}

\begin{lema} \label{lema4}
For $j\geq1$
the projection map to the first factor
\begin{equation}\label{pi1}
\pi_1: \HH^{2j}(\cG^{1*}_\BB) \longrightarrow \frac{\BB}{(f_1,f_2)},
\hskip 1cm \pi_1 (a,b) = a,
\end{equation}
induces an exact sequence:

\begin{equation} \label{H2j}
0 \longrightarrow \frac{\Ann_\BB(f_1,f_2)}
{\begin{pmatrix} c_{11} \cr c_{12} \cr \end{pmatrix}
\cdot
\frac{H^1(\cK_\BB(f_1,f_2)^*)}{
<\begin{pmatrix} -c_{12} \cr c_{11} \cr \end{pmatrix}>}}
\FlechaDer{i_2}
\HH^{2j}(\cG^{1*}_\BB)
\FlechaDer{\pi_1}
\frac{(0:
\begin{pmatrix} -c_{12} \cr c_{11} \cr \end{pmatrix})_{H^1(\cK_\BB(f_1,f_2)^*)}}{(f_1,f_2)}
\longrightarrow
0
\end{equation}
\end{lema}
\begin{proof}
$\begin{pmatrix} a \cr b \cr \end{pmatrix} \in \HH^{2j}(\cG^{1*}_\BB)$ satisfies $-ac_{12}+bf_1= 0$ and
$ac_{11}+bf_2=0$, or in matrix notation
$$
a\begin{pmatrix} -c_{12} \cr c_{11} \cr \end{pmatrix}
+b\begin{pmatrix} f_1 \cr f_2 \cr \end{pmatrix}=0\ \ \
\hbox{so that} \ \ \  a \in (\begin{pmatrix} f_{1} \cr f_{2} \cr \end{pmatrix}:
\begin{pmatrix} -c_{12} \cr c_{11} \cr \end{pmatrix}),$$ and if its cohomology class is 0
then
$\begin{pmatrix} a \cr b \cr \end{pmatrix} \in <\begin{pmatrix} -f_{2} \cr c_{11} \cr \end{pmatrix},
\begin{pmatrix} f_{1} \cr c_{12} \cr \end{pmatrix}>$ so that
$a \in (f_1,f_2)$. This shows that the map $\pi_1$ is well defined and surjective.

The elements   in the kernel of $\pi_1$  may be represented in the form $(\alpha_1f_1+\alpha_2f_2,b^\prime)^t$,
so that it has a cohomologous element of the form

$$\begin{pmatrix} 0 \cr b \cr \end{pmatrix}=
\begin{pmatrix} \alpha_1f_{1} + \alpha_2f_{2} \cr b^\prime \cr \end{pmatrix}-\alpha_1
\begin{pmatrix} f_{1} \cr c_{12} \cr \end{pmatrix}
+ \alpha_2
\begin{pmatrix} -f_{2} \cr c_{11} \cr \end{pmatrix}.$$
The element $(0,b)^t$ belongs to $\HH^{2j}(\cG^{1*}_\BB)$ if and only if $b\in \Ann_\BB(f_1,f_2)$.
It is the $0$ element if and only if

$$
\begin{pmatrix} 0 \cr b \cr \end{pmatrix}=\alpha
\begin{pmatrix} -f_{2} \cr c_{11} \cr \end{pmatrix}+\beta
\begin{pmatrix} f_{1} \cr c_{12} \cr \end{pmatrix}=
\begin{pmatrix}
(\alpha,\beta)^t\cdot (-f_2,f_1)^t  \cr
(\alpha,\beta)^t\cdot (c_{11},c_{12})^t  \cr \end{pmatrix}
 $$

Hence $(\alpha,\beta)^t\in <(-f_2,f_1)^t>^\perp$. Consider the map
$$\cdot\begin{pmatrix} c_{11} \cr c_{12} \cr \end{pmatrix} :
<\begin{pmatrix} -f_2 \cr f_1 \cr \end{pmatrix}>^\perp \longrightarrow \BB$$
which vanishes on
$
<\begin{pmatrix} f_1 \cr f_2 \cr \end{pmatrix},
\begin{pmatrix} -c_{12} \cr c_{11} \cr \end{pmatrix}>$.
Its image lies in $\Ann_\BB(f_1,f_2)$ and consists of those
elements which represent  0 in $\Ker(i_2)$.
\end{proof}

\section{The hyperhomology of the Gobelin $\cG^2_\BB$.}

\subsection{The complex $\tot(\cG^2_\BB)$.
}

The small Gobelin $\cG_\BB^{2}$ constructed from
\\
\begin{equation}
\begin{pmatrix}
c_{11} &  c_{12}\cr
c_{21} &  c_{22}\cr
\end{pmatrix}
\begin{pmatrix}f_1\cr
f_2\cr
\end{pmatrix}
=
0
\label{BMatrixIdentity2}
\end{equation}
is given by Figure 1 with $G=G_2$.
The total complex $\tot(\cG^{2}_\BB)$ is
\\
\begin{equation}
0 \xleftarrow{}
\BB^1
\xleftarrow{\left(
\begin{smallmatrix}{{f}}_{1}&
{{f}}_{{2}}\\
\end{smallmatrix}
\right)}
\BB^2
\xleftarrow{\left(
\begin{smallmatrix}{-{{f}}_{{2}}}&
{{c}}_{11}&
{{c}}_{{2}1}\\
{{f}}_{1}&
{{c}}_{1{2}}&
{{c}}_{{2}{2}}\\
\end{smallmatrix}
\right)}
\BB^3
\xleftarrow{\left(
\begin{smallmatrix}{-{{c}}_{1{2}}}&
{-{{c}}_{{2}{2}}}&
{{c}}_{11}&
{{c}}_{{2}1}\\
{{f}}_{1}&
0&
{{f}}_{{2}}&
0\\
0&
{{f}}_{1}&
0&
{{f}}_{{2}}\\
\end{smallmatrix}
\right)}
\BB^4
\xleftarrow{\left(
\begin{smallmatrix}{-{{f}}_{{2}}}&
0&
{{c}}_{11}&
{{c}}_{{2}1}&
0\\
0&
{-{{f}}_{{2}}}&
0&
{{c}}_{11}&
{{c}}_{{2}1}\\
{{f}}_{1}&
0&
{{c}}_{1{2}}&
{{c}}_{{2}{2}}&
0\\
0&
{{f}}_{1}&
0&
{{c}}_{1{2}}&
{{c}}_{{2}{2}}\\
\end{smallmatrix}
\right)}
\BB^5
\xleftarrow{}
\cdots
\label{hyperhomologysequenceG2}
\end{equation}
\begin{equation}
\label{formulas}
	\cdots
	\xleftarrow{}
	\BB^{2j-1}
	\xleftarrow{\varphi_j}
	\BB^{2j}
	\xleftarrow{\psi_j}
	\BB^{2j+1}
\xleftarrow{\varphi_j}
	\BB^{2j+2}
	\xleftarrow{}
	\cdots
\end{equation}
\noindent
\renewcommand{\arraystretch}{0.5}
\[
{\scriptsize \scriptstyle
	\varphi_j =
	\left(
	\begin{array}{*3{r@{\,}}r|*3{r@{\,}}r}
	-c_{12} & -c_{22}&&& c_{11} & c_{21}&& \\
	 &\ddots & \ddots && & \ddots & \ddots & \\
	& & -c_{12} & -c_{22}&& & c_{11} & c_{21}  \\
	&&&&&&\\\hline
	&&&&&&\\
	f_1 &&&& f_2 &\\
	& \ddots &&&& \ddots \\
	&& \ddots &&&& \ddots \\
	&&& f_1 &&&& f_2
	\end{array}
	\right), \quad
	\psi_j =
	\left(
	\begin{array}{*2{c@{\,}}c|c@{\,}c@{}c@{\,}c}
	-f_2 & & &c_{11} & c_{21}&& \\
	&\ddots & && \ddots & \ddots & \\
	&& -f_2 & && c_{11} & c_{21}  \\
	&&&&&&\\\hline
	&&&&&&\\
	f_1 &&& c_{12} & c_{22}&\\
	&\ddots&& & \ddots & \ddots & \\
	&&f_1& & & c_{12} & c_{22}
	\end{array}
	\right)}
\]

\subsection{Even dimensional hyperhomology $\HH_{2j}(\cG_\BB^2)$}

 Let $\alpha:=(a_1,\ldots,a_j,b_1,\ldots,b_{j+1})^t \in \BB^{2j+1}$
be a cycle  $\psi_j(\alpha)=0$, for $j\geq 1$. The cycle condition consists of $2j$ equalities,
and they may be organized into $j$ pairs of equations
by considering the $i$ and $i+j$ terms to obtain for $i=1,\ldots,j$:
\begin{equation}
a_i \begin{pmatrix}-f_2 \cr f_1 \end{pmatrix}
+b_i \begin{pmatrix}c_{11} \cr c_{12} \end{pmatrix}
+ b_{i+1} \begin{pmatrix}c_{21} \cr c_{22} \end{pmatrix} =
 \begin{pmatrix}0 \cr 0 \end{pmatrix}
\label{eq0}
\end{equation}
Considering these equations in $H_1(\KK_\BB(f_1,f_2))$
they become
\begin{equation}b_i \begin{pmatrix}c_{11} \cr c_{12} \end{pmatrix}
+ b_{i+1} \begin{pmatrix}c_{21} \cr c_{22} \end{pmatrix} =
 \begin{pmatrix}0 \cr 0 \end{pmatrix} \in H_1(\cK_\BB(f_1,f_2))
\label{eq1}
\end{equation}
This means that in order to solve (\ref{eq0}) we first solve (\ref{eq1}),
to obtain $(b_1,\ldots,b_{j+1})$, and then afterwards choose the $(a_1,\ldots,a_j)$
so that (\ref{eq0}) is satisfied. Such $a_i$ always exist and are unique up
to adding an element in $Ann_\BB(f_1,f_2)$. This and a direct inspection
of the columns of $\varphi_j$ give that the projection
$\rho_j(\alpha)=(b_1,\ldots,b_{j+1})$ induces the horizontally exact diagram:
$$
\begin{matrix}
0 & \rightarrow & \BB^j & \rightarrow &\BB^{2j+1} & \xrightarrow{\rho_j}& \BB^{j+1} & \rightarrow& 0 \cr
&& \cup && \cup && \cup  \cr
0 & \rightarrow & \Ann_\BB(f_1,f_2)^j & \rightarrow &\Ker(\psi_j) & \rightarrow& \rho_j(\Ker(\psi_j)) & \rightarrow& 0 \cr
&& \cup && \cup && \cup  \cr
0 & \rightarrow &  (\begin{pmatrix}-c_{12} \cr c_{11} \end{pmatrix}\cdot
\begin{pmatrix}f_1 \cr f_2 \end{pmatrix}^\perp
+\begin{pmatrix}-c_{22} \cr c_{21} \end{pmatrix} \cdot
\begin{pmatrix}f_1 \cr f_2 \end{pmatrix}^\perp
)^j & \rightarrow &Im(\varphi_j)  & \rightarrow& (f_1,f_2)^{j+1} & \rightarrow& 0 \cr
\end{matrix}
$$
In particular, we get a well defined map
$$\tilde\rho_j:\HH_{2j}(\cG^2_\BB) \rightarrow (\frac{\BB}{(f_1,f_2)})^{j+1},$$
which tells us that the $b_j$ components of the cycle $\alpha$
are invariants of the hyperhomology classes, when considered in the ring $\frac{\BB}{(f_1,f_2)}$.
Recall the definition of the flag of ideals (\ref{flag1})  and (\ref{flag2}) in
$ \BB$.

 \begin{prop} \label{prop1} Let $\alpha:=(a_1,\ldots,a_j,b_1,\ldots,b_{j+1})^t \in \BB^{2j+1}$
be a cycle  $\psi_j(\alpha)=0$ with $j\geq 1$, then: \vskip 2mm

{\rm a)} For $k=1,\ldots,j+1$ we have $b_k \in F_{k-1}^\prime \cap F_{j-k+1}  $; in particular $b_1 \in F_j $ and
$b_{j+1} \in F_j^\prime$.

{\rm b)} Given $b_1\in F_{j} $   we may construct a cycle $\alpha \in \BB^{2j+1}$
such that its $j+1$   component is $b_1$.

{\rm c)} Given $b_{j+1}\in F_{j}^\prime$ we may construct a cycle $\alpha \in \BB^{2j+1}$
such that its $2j+1$ component is $b_{j+1}$.

{\rm d)} The correspondence $b_1 \leftrightarrow b_{j+1}$ established between the components
of the $j$-cycles induces   isomorphisms

$$\Phi_j: \frac{F_j}{L_j} \longrightarrow \frac{F_j^\prime}{L_j^\prime}
$$
\end{prop}

\begin{proof} a) We have $b_{1}\in \BB = F_0^\prime$.
Equation (\ref{eq1}) with $i=1$ gives
$  b_2 \in
(\begin{pmatrix}c_{11} \cr c_{12} \end{pmatrix}:
\begin{pmatrix}c_{21} \cr c_{22} \end{pmatrix})_{H_1(\cK_\BB(f_1,f_2))}
=F_1^\prime$.
Now Equation (\ref{eq1}) with $i=2$ gives
$  b_3 \in
(F_1^\prime\begin{pmatrix}c_{11} \cr c_{12} \end{pmatrix}:
\begin{pmatrix}c_{21} \cr c_{21} \end{pmatrix})_{H_1(\cK_\BB(f_1,f_2))}
=F_2^\prime$, since we already know that $b_2\in F_1^\prime$.
Repeating the procedure for increasing
$i$ till we arrive for $i=j$ to the conclusion
$b_{j+1}\in F_{j-1}^\prime$.

Now beginning at the other extreme, we have
$b_{j+1}\in \BB = F_0 $.
Equation (\ref{eq1}) with $i=j$ gives
$  b_j \in
(\begin{pmatrix}c_{21} \cr c_{22} \end{pmatrix}:
\begin{pmatrix}c_{11} \cr c_{12} \end{pmatrix})_{H_1(\cK_\BB(f_1,f_2))}
=F_1$.
Now Equation (\ref{eq1}) with $i=j-1$ gives
$  b_{j-1} \in
(F_1\begin{pmatrix}c_{21} \cr c_{22} \end{pmatrix}:
\begin{pmatrix}c_{11} \cr c_{12} \end{pmatrix})_{H_1(\cK_\BB(f_1,f_2))}
=F_2$, since we already know that $b_j\in F_1 $.
Repeating the procedure for increasing
$i$ till we arrive for $i=1$ to the conclusion
$b_1\in F_{j} $.

b)  Now assume that one begins with $b_1\in F_j$,
then one obtains from the hypothesis a $b_2 \in F_{j-1}$
satisfying the
 equation (\ref{eq1}) with $i=1$. Now using the definition of
$F_{j-1}$ one obtains a $b_3 \in F_{j-2}$
satisfying the
correspondign equation (\ref{eq1}), and so on.
Now transporting these equations in $H_1(\cK_\BB(f_1,f_2))$
to $\BB^2$ one obtains the corresponding equations (\ref{eq0}) and
the elements $a_1,\ldots,a_j\in \BB$.
Grouping them together into a vector $\alpha$ the equations consist
exactly of the cycle condition $\psi_j(\alpha)=0$.

c) Similar to b).

d) Suppose that we have a cycle that has the form $(0,b_2,\ldots,b_{j+1})$.
Equation (\ref{eq1}) with $i=1$ gives $b_2 \in
\begin{pmatrix}0
:\begin{pmatrix}c_{21} \cr c_{22} \end{pmatrix}\end{pmatrix} =
L_1^\prime$.
Equation (\ref{eq1}) with $i=2$ gives $b_3 \in
L_2^\prime$ since we already have information on $b_2$, till one obtains that $b_{j+1} \in L_j^\prime$.
This shows that the map $\Phi_j$ is well defined. But by reversing the above procedure,
we obtain that $\Phi_j^{-1}$ is also well defined; hence $\Phi_j$ is an isomorphism
\end{proof}

 $\Phi_1$   gives the isomorphism

$$
\frac{\begin{pmatrix}\begin{pmatrix}c_{21} \cr c_{22} \end{pmatrix} :
\begin{pmatrix}c_{11} \cr c_{12} \end{pmatrix} \end{pmatrix}    }
{\begin{pmatrix}0 :
\begin{pmatrix}c_{11} \cr c_{12} \end{pmatrix}\end{pmatrix} }
\leftrightarrow
 \frac{\begin{pmatrix}\begin{pmatrix}c_{11} \cr c_{12} \end{pmatrix} :
\begin{pmatrix}c_{21} \cr c_{22} \end{pmatrix} \end{pmatrix}}
{\begin{pmatrix}0 :
\begin{pmatrix}c_{21} \cr c_{22} \end{pmatrix}\end{pmatrix} }
$$
that can be interpreted as being induce from the quotient of the 2 syzygies. Similarly,
the other isomorphisms $\Phi_j$ are induced from powers of this quotient.

\subsection{Even dimensional hypercohomology $\HH^{2j}( \cG^{2*}_\BB) $.   }

Consider the
  complex $\tot(\cG^{2}_\BB)^*$ obtained from
$(\ref{hyperhomologysequenceG2})$ by applying the functor
$\Hom_\BB(*,\BB)$. Since $\BB$ is a   Gorenstein $K$-algebra of
dimension 0 and the complex is by free $\BB$-modules,
the dual complex
has the form
\\
$$
0 \xrightarrow{}
\BB^1
\xrightarrow{\left(
\begin{smallmatrix}{{f}}_{1}&
{{f}}_{{2}} \\
\end{smallmatrix}
\right)^t}
\BB^2
\xrightarrow{\left(
\begin{smallmatrix}{-{{f}}_{{2}}}&
{{c}}_{11}&
{{c}}_{{2}1}\\
{{f}}_{1}&
{{c}}_{1{2}}&
{{c}}_{{2}{2}}\\
\end{smallmatrix}
\right)^t}
\BB^3
\xrightarrow{\varphi_2^t}
	\cdots
		\xrightarrow{\varphi_j^t}
	\BB^{2j}
	\xrightarrow{\psi_j^t}
	\BB^{2j+1}
\xrightarrow{\varphi_{j+1}^t}
	\BB^{2j+2}
	\xrightarrow{}
	\cdots
$$
\noindent
\renewcommand{\arraystretch}{0.5}

 \begin{prop} \label{prop2} Let $\beta:=(a_1,\ldots,a_j,b_1,\ldots,b_{j+1})^t \in \BB^{2j+1},
j\geq 1$,
be a cocycle   $\varphi_{j+1}^t(\beta)=0$, then:

\vskip 3mm
{\rm a)} For $i=1,\ldots,j $ we have $a_i \in L_{i} \cap L^\prime_{j+1  -i}$; in particular $a_1 \in L_1 \cap L_{j }^\prime $ and
$a_{j } \in L_j \cap L_1^\prime$.

\vskip 2mm
{\rm b)} Given $a_1\in L_{1} \cap L_j^\prime$, we may construct a cocycle $\beta \in \BB^{2j+1}$
such that its first component is $a_1$.

\vskip 2mm
{\rm c)} Given $a_{j} \in L_{j} \cap L_1^\prime$, we may construct a cocycle $\beta \in \BB^{2j+1}$
such that its $j$ component is $a_{j}$.

\vskip 2mm

{\rm d)} The correspondence $a_1 \leftrightarrow a_{j}$ established between the components
of the $j$-cocycles induces   isomorphisms

$$\Psi_j: \frac{L_1 \cap L_j^\prime}{L_1 \cap L_{j-1}^\prime} \longrightarrow
\frac{L_j \cap L_1^\prime}{L_{j-1} \cap L_{1}^\prime}
$$

\end{prop}

\begin{proof}  a) The cocycle condition,
which consists of $2j+2$ equalities, may be organized into $j+1$ pair of equations
by considering the $i$ and $i+j$ terms to obtain:
$$a_1 \begin{pmatrix}-c_{12} \cr c_{11} \end{pmatrix}
+b_1 \begin{pmatrix}f_{1} \cr f_{2} \end{pmatrix} =
 \begin{pmatrix}0 \cr 0 \end{pmatrix}
\ \ , \ \
a_j \begin{pmatrix}-c_{22} \cr c_{21} \end{pmatrix}
+b_{j+1} \begin{pmatrix}f_{1} \cr f_{2} \end{pmatrix} =
 \begin{pmatrix}0 \cr 0 \end{pmatrix}
$$
 and for $i=1,\ldots,j-1$:
$$a_i \begin{pmatrix}-c_{22} \cr c_{21} \end{pmatrix}
+a_{i+1} \begin{pmatrix}-c_{12} \cr c_{11} \end{pmatrix}
+ b_{i+1} \begin{pmatrix}f_{1} \cr f_{2} \end{pmatrix} =
 \begin{pmatrix}0 \cr 0 \end{pmatrix}$$
Considering these equations in
$$H^1(\KK_\BB(f_1,f_2) ^*)=\frac{\begin{pmatrix}-f_{2} \cr f_{1} \end{pmatrix}^\perp}
{\begin{pmatrix}f_{1} \cr f_{2} \end{pmatrix}}$$
they become for $i=1,\ldots,j-1$:
\begin{equation} \label{cocycle}
a_1 \begin{pmatrix}-c_{12} \cr c_{11} \end{pmatrix} =
 \begin{pmatrix}0 \cr 0 \end{pmatrix} \ \ , \ \
a_i \begin{pmatrix}-c_{22} \cr c_{21} \end{pmatrix}
+a_{i+1} \begin{pmatrix}-c_{12} \cr c_{11} \end{pmatrix} =
 \begin{pmatrix}0 \cr 0 \end{pmatrix} \ \   ,\ \
a_j \begin{pmatrix}-c_{22} \cr c_{21} \end{pmatrix} =
 \begin{pmatrix}0 \cr 0 \end{pmatrix}
\end{equation}
We note that the $\BB$-involution $\kappa:\BB^2 \longrightarrow \BB^2$
in (\ref{kappa})
has the property that for  $\tau ,\tau^\prime \in \BB^2$
and any ideal $I \subset \BB$ we have
$
(I\tau:\tau^\prime) = (I\kappa(\tau):\kappa(\tau^\prime))
$
and hence the flags constructed with the equation (\ref{syzygyreduced})
coincide with the flags constructed with the equation
\begin{equation}
\begin{pmatrix} -c_{12}&c_{11}\cr
-c_{22}&c_{21}\cr
\end{pmatrix}
\begin{pmatrix} -f_2\cr
f_1\cr
\end{pmatrix} = \begin{pmatrix} 0\cr
0\cr
\end{pmatrix}.
\end{equation}
The cocycle conditions (\ref{cocycle}) are then:
$$a_{1}\in
(\begin{pmatrix}0 \cr 0 \end{pmatrix}:
\begin{pmatrix}-c_{12} \cr c_{11} \end{pmatrix})=L_1  \hskip 2cm,
\hskip2cm
a_{j}\in
(\begin{pmatrix}0 \cr 0 \end{pmatrix}:
\begin{pmatrix}-c_{22} \cr c_{21} \end{pmatrix})=L_1^\prime$$
and if we begin to solve (*) with $i=1$ till i = j - 1:
$$a_{i+1} \in (L_{i} \begin{pmatrix}-c_{22} \cr c_{21} \end{pmatrix}:
\begin{pmatrix}-c_{12} \cr c_{11} \end{pmatrix}) = L_{i+1 }     $$
and if we begin to solve (*) with $i=j-1$ till i= 1:

$$
a_{i }\in
(L^\prime_{j-i}\begin{pmatrix}-c_{12} \cr c_{22} \end{pmatrix}:
\begin{pmatrix}-c_{22} \cr c_{21} \end{pmatrix})=L_{j+1-i }^\prime$$

\hskip 2mm
b) and c) Similar to the proofs of b)and c) in Proposition \ref{prop1}.\\

\hskip 2mm
d) If we have 2 $j$-cocycles with the same first component $a_1$, by substracting them,
we obtain a $j-1$-cocycle, whose last component will be in
$L_j \cap L^\prime_{j}$, hence $\Psi_j(a_1)$ is well defined modulo $L_j \cap L^\prime_{j}$,
Reversing the argument by beginning  with   the last component $a_j$,
we obtain $\Psi^{-1}$.
\end{proof}

\subsection{The long exact sequence of hyperhomology of the small Gobelin}

 In this case the exact sequence of complexes in
Theorem \ref{ExactaGobelinos} is Figure 2, where

\begin{figure}[h]
 {\scriptsize
$$
\begin{array}{*{30}{c@{\,}}}

&&
&&
& 0 &
&& 0 &
&& 0 &
&& 0 &
&& 0 &
&& &
&& 0 &
&& 0 &
&&
\\

&&
&&
& \downarrow &
&&\downarrow &
&&\downarrow &
&&\downarrow &
&&\downarrow &
&& &
&&\downarrow &
&&\downarrow &
&&
\\

&
& 0 & \FlechaIzq{} &
& \BB  & \FlechaIzq{(f_{1},f_{2})}&
& \BB^2  & \FlechaIzq{C_{\psi }}&
& \BB^2  & \FlechaIzq{C_{\varphi }}&
& \BB^2  & \FlechaIzq{C_{\psi }}&
& \BB^2  & \FlechaIzq{C_{\varphi }}&
& \cdots & \FlechaIzq{C_{\varphi }}&
& \BB^2  & \FlechaIzq{C_{\psi }}&
& \BB^2  & \FlechaIzq{C_{\varphi }}&
&\cdots
\\

&&
&&
& \downarrow \iota_0 &
&&\downarrow \iota_1 &
&&\downarrow \iota_2 &
&&\downarrow \iota_3 &
&&\downarrow \iota_4 &
&& &
&&\downarrow \iota_{2j-1} &
&&\downarrow \iota_{2j}  &
\\

&
& 0 & \FlechaIzq{}&
& \BB   & \FlechaIzq{(f_{1},f_{2})}&
& \BB^2 & \FlechaIzq{\psi_{1}}&
& \BB^3 & \FlechaIzq{\varphi_{1}}&
& \BB^4 & \FlechaIzq{\psi_{2}} &
& \BB^5 & \FlechaIzq{\varphi_{2}}&
& \cdots & \FlechaIzq{\varphi_{j}}&
& \BB^{2j} & \FlechaIzq{\psi_{j}}&
& \BB^{2j+1} & \FlechaIzq{\varphi_{j+1}}&
&\cdots
\\
&&
&&
& \downarrow  &
&&\downarrow  &
&&\downarrow \sigma^*_2 &
&&\downarrow \sigma^*_3 &
&&\downarrow \sigma^*_4 &
&& &
&&\downarrow \sigma^*_{2j-1} &
&&\downarrow \sigma^*_{2j} &
\\

&
& 0 & \FlechaIzq{}&
& 0 & \FlechaIzq{} &
& 0 & \FlechaIzq{}&
& \BB   & \FlechaIzq{(f_{1},f_{2})}&
& \BB^2 & \FlechaIzq{\psi_{1}}&
& \BB^3 & \FlechaIzq{\varphi_{2}}&
& \cdots & \FlechaIzq{\varphi_{j-1}}&
& \BB^{2j-2} & \FlechaIzq{\psi_{j-1}}&
& \BB^{2j-1} & \FlechaIzq{\varphi_{j}}&
&\cdots
\\
&&
&
&&
&&
&&
&&\downarrow  &
&&\downarrow  &
&&\downarrow  &
&& &
&&\downarrow  &
&&\downarrow  &
\\
&&
&&
&&
&&&
&& 0 &
&& 0 &
&& 0 &
&&  &
&& 0 &
&& 0 &
\\

\end{array}
$$
}
\begin{center}
\parbox{12cm}{
 \caption{The short exact sequence of complexes.} \label{DiagramaGobelinos}
 }
\end{center}
\end{figure}
$$\iota_{2j}\begin{pmatrix} a \cr b \cr \end{pmatrix}
=
 \begin{pmatrix} a \cr 0_{j-1} \cr b \cr 0_{j}
 \end{pmatrix} \in \BB^{2j+1},
\hskip 5mm\hskip 5mm
\iota_{2j+1}\begin{pmatrix} a \cr b \cr \end{pmatrix}
=
 \begin{pmatrix} a \cr 0_{j} \cr b \cr
0_{j} \end{pmatrix}  \in \BB^{2j+2},
\hskip 5mm\hskip 5mm
\sigma_j^*\begin{pmatrix} a_1 \cr a_{2} \cr \vdots \cr b_1
\cr b_{2} \cr \vdots
 \end{pmatrix} =
\begin{pmatrix}   a_{2} \cr a_{3} \cr \vdots \cr
 b_{2} \cr b_{3} \cr\vdots
 \end{pmatrix}
$$
where $0_j$ denotes the $0$ column vector of size $j$,
then the boundary maps
$\partial_k:\HH_{k-1}(\cG^{2}_\BB)
\xrightarrow{}
\HH_{k}(\cG^{1}_\BB)
$ are induced from the maps:
\begin{equation}
\label{frontera1}
\partial_{2j+1}: \begin{pmatrix} a_1  \cr \vdots \cr a_j \cr b_1
 \cr \vdots \cr b_{j+1}
 \end{pmatrix} \xrightarrow {{\sigma_{2j+1}^*}^{-1}}
  \begin{pmatrix} 0\cr a_1  \cr \vdots \cr a_j \cr 0\cr b_1
 \cr \vdots \cr b_{j+1}
 \end{pmatrix} \xrightarrow{\psi_{j+1}}
b_1\begin{pmatrix}   c_{21} \cr 0_j\cr  c_{22} \cr 0_j
 \end{pmatrix} \xrightarrow {\iota_{2j}^{-1}}
b_1\begin{pmatrix}   c_{21} \cr  c_{22}
 \end{pmatrix} \in \BB^2
\end{equation}
\begin{equation}
\label{frontera2}
\partial_{2j}: \begin{pmatrix} a_1  \cr \vdots \cr a_j \cr b_1
 \cr \vdots \cr b_{j}
 \end{pmatrix} \xrightarrow {{\sigma_{2j}^*}^{-1}}
  \begin{pmatrix} 0\cr a_1  \cr \vdots \cr a_j \cr 0\cr b_1
 \cr \vdots \cr b_{j}
 \end{pmatrix} \xrightarrow {\varphi_{j+1}}
 \begin{pmatrix}   -c_{22}a_1+  c_{21}b_1 \cr 0_{2j}  \cr
 \end{pmatrix}
\xrightarrow {\iota_{2j}^{-1}}
 \begin{pmatrix}
 \begin{pmatrix}  -c_{22} \cr  c_{21} \end{pmatrix} \cdot
\begin{pmatrix}    a_{1} \cr  b_{1}  \end{pmatrix}
 \cr 0
 \end{pmatrix}
\in \BB^2
\end{equation}
 The long exact sequence of hyperhomology groups
of the exact sequence of double complexes in Figure \ref{DiagramaGobelinos} is:
$$
0\xleftarrow{}
\HH_{0}(\cG^{2}_\BB)
\xleftarrow{i_0}
\HH_{0}(\cG^{1}_\BB)
\xleftarrow{\partial_{0}}
0\xleftarrow{\sigma_1^*}
\HH_{1}(\cG^{2}_\BB)
\xleftarrow{i_1}
\HH_{1}(\cG^{1}_\BB)
\xleftarrow{\partial_1}
\HH_{0}(\cG^{2}_\BB)
\xleftarrow{\sigma_2^*}
\cdots
$$
\begin{equation}
\xleftarrow{\sigma_{2j-1}^*}
\HH_{2j-1}(\cG^{2}_\BB)
 \xleftarrow{i_{2j-1}}
\HH_{2j-1}(\cG^{1}_\BB)
\xleftarrow{\partial_{2j-1}}
\HH_{2j-2}(\cG_\BB^2)
\xleftarrow{\sigma_{2j}^*}
\HH_{2j}(\cG^{2}_\BB)
\xleftarrow{i_{2j}}
\label{HomGob3}
\end{equation}
\begin{equation}
\xleftarrow{}
\HH_{2j}(\cG^{1}_\BB)
\xleftarrow{\partial_{2j}}
\HH_{2j-1}(\cG^{2}_\BB)
\xleftarrow{\sigma_{2j+1}}
\HH_{2j+1}(\cG^{2}_\BB)
\xleftarrow{i_{2j+1}}
\HH_{2j+1}(\cG^{1}_\BB)
\xleftarrow{\partial_{2j+1}}
\cdots
\label{HomGob4}
\end{equation}

For $j\geq1$,  the map
$$
\partial_{2j+1}:\HH_{2j}(\cG^{2}_\BB)
\longrightarrow \HH_{2j+1}(\cG^{1}_\BB) =
\frac
{<\begin{pmatrix}f_{1} \cr f_{2}\end{pmatrix}
,\begin{pmatrix}-c_{12} \cr c_{11}\end{pmatrix}
>^\perp}
{<\begin{pmatrix}-f_{2} \cr f_{1}\end{pmatrix},
\begin{pmatrix}c_{11} \cr c_{12}\end{pmatrix}>},
\hskip 5mm\hskip 5mm
\partial_{2j+1} \begin{pmatrix} a_1  \cr \vdots \cr a_j \cr b_1
 \cr \vdots \cr b_{j+1}
 \end{pmatrix} =
b_1\begin{pmatrix}   c_{21} \cr  c_{22}
 \end{pmatrix}
$$
has image $F_j \begin{pmatrix}   c_{21} \cr  c_{22}
 \end{pmatrix}$ by Proposition \ref{prop1}.
The map
$$\begin{pmatrix}c_{21} \cr c_{22}\end{pmatrix}:
{F_j} \longrightarrow \HH_{2j+1}(\cG^{1}_\BB)$$
has kernel $F_j \cap F_1^\prime$ and
hence the  map
$\begin{pmatrix}c_{21} \cr c_{22}\end{pmatrix}:
\frac{F_j}{F_j \cap F_1^\prime } \longrightarrow \HH_{2j+1}(\cG^{1}_\BB)$
is injective.
Hence both maps have the same image and we obtain:

\begin{lema}  \label{lema4a} For $j\geq1$ the image of $\partial_{2j+1}$ is isomorphic to
$\frac{F_j}{F_j \cap F_1^\prime }$.\end{lema}

\begin{prop} \label{prop3}
For $j\geq1$ we have an exact sequence:

$$
0 \xleftarrow{}  \frac{F_{j-1}}{F_{j-1}\cap F_1^\prime}
\xleftarrow{}
\HH_{2j-2}(\cG^{2}_\BB)
\xleftarrow{\sigma_{2j}^*}
\HH_{2j}(\cG^{2}_\BB)
\xleftarrow{i_{2j}}
\HH_{2j}(\cG^{1}_\BB)
\xleftarrow{\partial_{2j}}
$$
\begin{equation}\label{HomGob5}
\xleftarrow{\partial_{2j}}
\HH_{2j-1}(\cG^{2}_\BB)
\xleftarrow{\sigma_{2j+1}^*}
\HH_{2j+1}(\cG^{2}_\BB)))
\xleftarrow{i_{2j+1}}
\frac{\HH_{2j+1}(\cG^{1}_\BB)
}{\frac{F_j}{F_j \cap F_1^\prime}\begin{pmatrix}c_{21} \cr c_{22}\end{pmatrix}}
\xleftarrow{}
0
\end{equation}
\end{prop}
\begin{proof}
We are incorporating into the long exact sequence (\ref{HomGob3}) and (\ref{HomGob4})
the conclusion of Lemma \ref{lema4a}. On the righthand side we are putting
$\coker(\partial_{2j+1})$
and on the lefthand side we are putting  $\Imag(\partial_{2j-1})$.
\end{proof}

\subsection{The long exact sequence of hypercohomology of the dual small Gobelin}

Consider the short exact sequence of complexes obtained by applying the functor
$\Hom_\BB(*,\BB)$ and the long exact sequence of hypercohomology groups. Since $\BB$ is
a Gorenstein algebra of dimension 0, this long exact sequence is the dual of the
sequence (\ref{HomGob3}) and (\ref{HomGob4})
and in particular it contains the dual of (\ref{HomGob5}):

$$
0 \xrightarrow{}  [\frac{F_{j-1}}{F_{j-1}\cap F_1^\prime}]^*
\xrightarrow{}
\HH^{2j-2}(\cG^{2*}_\BB)
\xrightarrow{\sigma_{2j} }
\HH^{2j}(\cG^{2*}_\BB)
\xrightarrow{i_{2j}^*}
\HH^{2j}(\cG^{1*}_\BB)
\xrightarrow{\partial_{2j}^*}
$$
\begin{equation}\label{HomGob6}
\xrightarrow{\partial_{2j}^*}
\HH^{2j-1}(\cG^{2*}_\BB)
\xrightarrow{\sigma_{2j+1}}
\HH^{2j+1}(\cG^{2*}_\BB)
\xrightarrow{i_{2j+1}^*}
[\frac{\HH^{2j+1}(\cG^{1}_\BB)
}{\frac{F_j}{F_j \cap F_1^\prime}\begin{pmatrix}c_{21} \cr c_{22}\end{pmatrix}}]^*
\xrightarrow{}
0
\end{equation}

We want to compute
$$\Imag(\partial_{2j}^*) = \Ker (\sigma_{2j+1}) \simeq \coker(i_{2j}^*)$$

\begin{lema} \label{lema5} For $j\geq 1$ we have
$$\Imag(i_{2j}^*) = \pi_1^{-1}(L_j\cap L_1^\prime)$$
where $\pi_1$ is the projection to the first factor in (\ref{H2j})
and $L_1^\prime$ and $L_j$ are defined in (\ref{flag1}) and
(\ref{flag2}).
\end{lema}
\begin{proof}
We have $i_{2j}^*(a_1,\ldots,b_1,\ldots,b_{j+1})^t=(a_1,b_1)^t$.
By Proposition \ref{prop2} the only condition on $a_1$ so that it is part of a
cocycle is $a_1 \in L_j \cap L_1^\prime$, and since the  only condition on $b_1$:
$$ a_1\begin{pmatrix}-c_{12} \cr c_{11}\end{pmatrix}+b_1
\begin{pmatrix}f_{1} \cr f_{2}\end{pmatrix}
= \begin{pmatrix} 0 \cr 0\end{pmatrix}$$
is the same in $\HH^{2j}(\cG^{2*}_\BB)
$ and in $\HH^{2j}(\cG^{1*}_\BB)$, we have that all admissible $b_1$ are allowed,
so the lemma follows from Proposition \ref{lema4}.
\end{proof}

\subsection{Proof of Theorem \ref{CDos}}

\begin{proof}
Now that we have computed $\Imag(i_{2j}^*)$, we may split the exact sequence
(\ref{HomGob6}) into 2 exact sequences:

$$
0 \xrightarrow{}  [\frac{F_{j-1}}{F_{j-1}\cap F_1^\prime}]^*
\xrightarrow{}
\HH^{2j-2}(\cG^{2*}_\BB)
\xrightarrow{\sigma_{2j} }
\HH^{2j}(\cG^{2*}_\BB)
\xrightarrow{i_{2j}^*}
\pi_1^{-1}(L_j \cap L_1^\prime)
\xrightarrow{}
0
$$
\begin{equation}\label{HomGob8}
0 \xrightarrow{}
\frac{\HH^{2j}(\cG^{1*}_\BB)}{\pi_1^{-1}(L_j \cap L_1^\prime)}
\xrightarrow{\partial_{2j}^*}
\HH^{2j-1}(\cG^{2*}_\BB)
\xrightarrow{\sigma_{2j+1}}
\HH^{2j+1}(\cG^{2*}_\BB)
\xrightarrow{i_{2j+1}^*}
[\frac{\HH^{2j+1}(\cG^{1}_\BB)
}{\frac{F_j}{F_j \cap F_1^\prime}\begin{pmatrix}c_{21} \cr c_{22}\end{pmatrix}}]^*
\xrightarrow{}
0
\end{equation}

\vskip 3mm
Consider the first term in (\ref{HomGob8}).
On using the definition of $L_1^\prime$ in (\ref{flag2}) and
Lemma \ref{lema4}, the projection
$\pi_1$ to the first component (\ref{pi1}) gives
an exact sequence:

$$\HH^{2j}(\cG^{1*}_\BB) \xrightarrow{\pi_1} L_1^\prime \longrightarrow 0$$
and hence it induces an isomorphism
$$\frac{\HH^{2j}(\cG^{1*}_\BB)}{\pi_1^{-1}(L_j \cap L_1^\prime)}
 \xrightarrow{\pi_1} \frac{L_1^\prime}{L_j \cap L_1^\prime}
$$

We may then write the exact sequence (\ref{HomGob8}) as:
\begin{equation}\label{HomGob8bis}
0 \xrightarrow{}
 \frac{ L_1^\prime}{ L_j \cap  L_1^\prime}
\xrightarrow{}
\HH^{2j-1}(\cG^{2*}_\BB)
\xrightarrow{\sigma_{2j+1}}
\HH^{2j+1}(\cG^{2*}_\BB)
\xrightarrow{i_{2j+1}^*}
[\frac{\HH^{2j+1}(\cG^{1}_\BB)
}{\frac{F_j}{F_j \cap F_1^\prime}\begin{pmatrix}c_{21} \cr c_{22}\end{pmatrix}}]^*
\xrightarrow{}
0
\end{equation}
Taking duals of the exact sequence
$$0 \longrightarrow  L_j\cap L_1^\prime
\longrightarrow L_1^\prime \longrightarrow \frac{ L_1^\prime}{ L_j \cap L_1^\prime}
\longrightarrow 0$$
and using the non-degenerate bilinear pairing
(\ref{d1}) of the Gorenstein algebra $\BB$,
we obtain the exact sequence
$$0 \longleftarrow \frac{\BB}{\Ann_\BB(L_j\cap L_1^\prime)}
 \longleftarrow \frac{\BB}{\Ann_\BB(L_1^\prime)}
 \longleftarrow [\frac{ L_1^\prime}{L_j \cap  L_1^\prime}]^*
\longleftarrow 0$$
hence
$$[\frac{L_1^\prime}{ L_j \cap  L_1^\prime}]^*
= \frac{\Ann_\BB( L_j\cap L_1^\prime)}{\Ann_\BB( L_1^\prime)}.$$
Taking duals of the sequence (\ref{HomGob8bis}), we
obtain:

\begin{equation}\label{HomGob8a}
0 \xleftarrow{}
\frac{\Ann_\BB( L_j\cap L_1^\prime)}{\Ann_\BB(L_1^\prime)}
\xleftarrow{}
\HH_{2j-1}(\cG^{2}_\BB)
\xleftarrow{\sigma_{2j+1}^*}
\HH_{2j+1}(\cG^{2}_\BB)
\xleftarrow{i_{2j+1}}
 \frac{\HH_{2j+1}(\cG^{1}_\BB)
}{\frac{F_j}{F_j \cap F_1^\prime}\begin{pmatrix}c_{21} \cr c_{22}\end{pmatrix}}
\xleftarrow{}
0
\end{equation}

Using the exactness of the sequences (\ref{HomGob5}) and (\ref{HomGob8a}), we
obtain the exact sequence

$$
0 \xleftarrow{}  \frac{F_{j-1}}{F_{j-1}\cap F_1^\prime}
\xleftarrow{}
\HH_{2j-2}(\cG^{2}_\BB)
\xleftarrow{\sigma_{2j}^* }
\HH_{2j}(\cG^{1}_\BB)
\xleftarrow{i_{2j}}
\frac{\HH_{2j}(\cG^{1}_\BB)}{\frac{\Ann_\BB( L_j\cap  L_1^\prime)}{\Ann_\BB( L_1^\prime)}
\begin{pmatrix}1 \cr 0\end{pmatrix}}
\xleftarrow{}
0
$$
\end{proof}

\begin{Cor}\label{Cor1}
Let
 $$ \begin{pmatrix}c_{11} \cr c_{12}\end{pmatrix},
\begin{pmatrix}c_{21} \cr c_{22}\end{pmatrix};
\begin{pmatrix}d_{11}   \cr d_{12}  \end{pmatrix},
\begin{pmatrix}d_{11}   \cr d_{12}  \end{pmatrix} \in
<\begin{pmatrix}f_{1} \cr f_{2}\end{pmatrix}>^\perp\subset\BB^2
\ \ , \ \
\begin{pmatrix}c_{11} \cr c_{12}\end{pmatrix} \equiv
\begin{pmatrix}d_{11}  \cr d_{12}  \end{pmatrix},
\begin{pmatrix}c_{21} \cr c_{22}\end{pmatrix} \equiv
\begin{pmatrix}d_{21}   \cr d_{22}  \end{pmatrix} \in
H_{1}(\cK_\BB(f_1,f_2))
$$
be 2 pairs of syzygies which are homologous, then the corresponding small Gobelins constructed with them
have equidimensional  hyperhomology groups.
\end{Cor}
\begin{proof}
  The flags constructed for each pair of syzygies are the same, since they depend only
  on their homology class in  $\HH_{1}(\cK_\BB(f_1,f_2))$, and hence the hyperhomology
  groups are equidimensional by Theorem \ref{CDos}.
\end{proof}

\section{Properties of the flags} \label{PropertiesFlag}

\begin{lema}\label{PuntoEstable}$ $\\
\hbox{{\rm a)}} Let  $I\subset \BB$ be an ideal such that
$I =(I\bigl( \begin{smallmatrix}  c_{21} \cr c_{22} \end{smallmatrix} \bigr)  :
\bigl( \begin{smallmatrix} c_{11} \cr c_{12}\end{smallmatrix} \bigr) )_{H_1 (\cK_{\BB}(f_1 , f_2))}.$
Then for every $j\geq 0$ we have $L_j\subset  I\subset F_j$  . $F_{\infty}$ is the largest of such class of ideals and $L_\infty$ is the smallest.
\\
\\
\hbox{{\rm b)}} Let $J\subset \BB$ be an ideal such that
$J=( J\bigl( \begin{smallmatrix}  c_{11} \cr c_{12}\end{smallmatrix} \bigr)  :
\bigl( \begin{smallmatrix} c_{21} \cr c_{22}\end{smallmatrix} \bigr)
)_{H_1(\cK_\BB(f_1,f_2))}. $
Then for every $j\geq 0$ we have  $L^\prime_j\subset  J\subset F^\prime_j$.
  $F^\prime_{\infty}$ is the largest of such class of ideals and $L_\infty^\prime$ is the smallest.
\end{lema}
\begin{proof}
Let $\phi_1,\phi_2:\BB\to H_1(\cK_\BB(f_1,f_2))$ be the morphisms defined by $$\phi_1(a)= a\bigl(\begin{smallmatrix}  c_{11} \cr c_{12}\end{smallmatrix}\bigr), \hskip 1cm \hbox{and}\hskip 1cm 
\phi_2(a)= a\bigl(\begin{smallmatrix}  c_{21} \cr c_{22}\end{smallmatrix}\bigr).$$
Note that $$F_j=\phi_1^{-1}\phi_2(F_{j-1}), \hskip 1cm
L _j=\phi_1^{-1}\phi_2(L _{j-1}), \hskip 1cm \hbox{and}  \hskip 1cm
 0 = L _0 \subset I=\phi_1^{-1}\phi_2(I)\subset F_0=\BB.$$
We will prove the statement by induction. Clearly, $I\subset \BB=F_0.$ If $I\subset F_{j-1}$, then
$I=\phi_1^{-1}\phi_2(I)\subseteq \phi_1^{-1}\phi_2(F_{j-1})=F_{j}$. Hence, $I\subset F_j$ for every $j\geq 0$ and the result follows.

The proof for $L _j$ is analogous because $L _0=0\subset I.$
If $L _{j-1}\subset I$, then
$L _j=\phi_1^{-1}\phi_2(L _{j-1})\subset I=\phi_1^{-1}\phi_2(I)$. Hence, $L_j  \subset I$ for every $j\geq 0$.
\end{proof}

\begin{Rem}
Lemma  \ref{PuntoEstable}   implies that the chains of ideals stabilize at the first place where there is no  proper containment: If $F_j=F_{j+1}$ then $F_\infty=F_j.$
Since $f_1$ and $f_2$ belong to the ideals $F_j,F^\prime_j,L_j$ and $L^\prime_j$ for $j\geq 1,$ we have that the length of all the flags of ideals is smaller than
or equal to $\dim_\KK \BB/(f_1,f_2).$
\end{Rem}

\begin{lema}\label{TL}
Let $T_1,T_2:V\to W$ be linear transformations of finite dimensional $\KK$-vector spaces and $V_2\subset V_1$ be linear subspaces of $V$, then
$$\hbox{{\rm codim}}(V_2,V_1) \geq \hbox{{\rm codim}}(T_1^{-1}T_2(V_2),T_1^{-1}T_2(V_1)).$$
\end{lema}
\begin{proof}
We have an induced well-defined surjective map 
$T_2:\frac{V_1}{V_2}\to \frac{T_2(V_1)}{T_2(V_2)}.$
Then, 
\begin{equation}\label{EqCodim1}
\dim_{\KK} V_{1} - \dim_{\KK}V_{2} \geq \dim_\KK T_2(V_1)-\dim_\KK T_2(V_2).
\end{equation}
We also have an induced well-defined injective  map
$T_1: \frac{T^{-1}_1 T_2(V_1)}{T^{-1}_1 T_2(V_2)} \to \frac{T_2(V_1)}{T_2(V_2)} ,$
which gives
\begin{equation}\label{EqCodim2}
\dim_\KK T_2(V_1)-\dim_\KK T_2(V_2)\geq 
\dim_\KK T^{-1}_1 T_2(V_1)-\dim_\KK T^{-1}_1 T_2(V_2).
\end{equation}
The desired inequality follows from (\ref{EqCodim1}) and (\ref{EqCodim2}).
\end{proof}

\begin{prop}\label{graduado} For $j\geq 0$ we have:
  $$\dim_{\KK}\frac{F_{j}}{F_{j+1}}\geq \dim_{\KK} \frac{F_{j+1}}{F_{j+2}},\hskip 1cm  \hskip 1cm
  \dim_{\KK}\frac{L_{j+2}}{L_{j+1}}\leq \dim_{\KK} \frac{L_{j+1}}{L_{j}}$$
 $$\dim_{\KK}\frac{F^\prime_{j}}{F^\prime_{j+1}}\geq \dim_{\KK} \frac{F^\prime_{j+1}}{F^\prime_{j+2}},\hskip 1cm  \hbox{and}\hskip 1cm
  \dim_{\KK}\frac{L^\prime_{j+2}}{L^\prime_{j+1}}\leq \dim_{\KK} \frac{L^\prime_{j+1}}{L^\prime_{j}}.$$
\end{prop}
\begin{proof}
Let $\phi_1:B\to H_1(\cK_\BB(f_1,f_2))$  be the morphism defined by $a\to a\bigl(\begin{smallmatrix}  c_{11} \cr c_{12}\end{smallmatrix}\bigr)$ and let
$\phi_2:B\to H_1(\cK_\BB(f_1,f_2))$ be the morphism defined by $a\to a\bigl(\begin{smallmatrix}  c_{21} \cr c_{22}\end{smallmatrix}\bigr)$.
We notice that $F_j=\phi_1^{-1}\phi_2 (F_{j-1})$.
Hence,
\begin{align*}
\dim_{\KK}F_{j+1}/F_{j} & =\dim_{\KK}F_{j+1}-\dim_\KK F_{j}
 = \dim_\KK \phi_1^{-1} \phi_2(F_{j})-\dim_\KK\phi_1^{-1}\phi_2(F_{j-1})\\
&\leq \dim_\KK F_j-\dim_\KK F_{j-1} .
\end{align*}
by Lemma \ref{TL}. Similarly, we obtain the statement about $F^\prime_j$ by interchanging the role of $\phi_1$ and $\phi_2.$
\end{proof}

\vskip 3mm
\section{Examples}

Recall from Section \ref{Sec3} the notation  $\mu := \dim_\KK\BB$ and $\nu := \dim_\KK\frac{\BB}{(f_1,f_2)}$.

 \subsection{Both syzygies are $0$}

If $\bigl( \begin{smallmatrix} c_{11} \cr c_{12}\end{smallmatrix} \bigr)
=\bigl( \begin{smallmatrix} c_{21} \cr c_{22}\end{smallmatrix} \bigr)
=\bigl( \begin{smallmatrix} 0 \cr 0\end{smallmatrix} \bigr) \in <\bigl( \begin{smallmatrix} f_1 \cr f_2 \end{smallmatrix} \bigr)>^\perp \subset \BB^2,
$ then both flags (\ref{flag1}) and (\ref{flag2}) are identical, and have the form:
$$
L_0   = 0 \subset L_1  =   \ldots =  L_\infty  = F_\infty =  \ldots  = F_0  =  \BB $$
From Proposition \ref{lema2} we obtain that for $j \geq1$:
\begin{equation}\label{a1}
\HH_0(\cG^1_\BB)= \frac{\BB}{(f_1,f_2)}\hskip 0.25cm,\hskip 0.25cm
\HH_{2j-1}(\cG^1_\BB)=  H_{ 1}(\cK_\BB(f_1,f_2)) \hskip 0.25cm,\hskip 0.25cm
\HH_{2j}(\cG^1_\BB)=  \Ann_\BB(f_1,f_2) \oplus \frac{\BB}{(f_1,f_2)},
\end{equation}
$\dim\HH_0(\cG^1_\BB) = \nu$ and $\dim\HH_j(\cG^1_\BB) = 2 \nu$  for $j\geq 1$,
on using (\ref{Koszul}).
On considering the exact sequences (\ref{teo1}) and (\ref{teo2}), we see that for $j\geq 1$ the left terms are 0,
and the right terms are $\HH_*(\cG^1_\BB)$. From this, one obtains readily
that   $\dim_\KK\HH_{j} (\cG^2_\BB)=  (j+1)\nu.$

Hence, the algorithm to construct a $j+2$ cycle from a $j$ cycle in this case,
imposes no conditions on the $j$-cycle, and we have freedom in choosing
in a $2\nu$-dimensional manner. This number bounds the dimension of all hyperhomology groups
of small Gobelins with $f_1$ and $f_2$ fixed, and with variable pair of syzygies.

If the syzygies satisfy
$\bigl( \begin{smallmatrix} c_{11} \cr c_{12}\end{smallmatrix} \bigr) = g_1 \bigl( \begin{smallmatrix} -f_2 \cr f_1 \end{smallmatrix} \bigr)\ \ ,\ \
\bigl( \begin{smallmatrix} c_{21} \cr c_{22}\end{smallmatrix} \bigr) = g_2 \bigl( \begin{smallmatrix} -f_2 \cr f_1 \end{smallmatrix} \bigr),
$
for $g_j \in \BB$, or equivalently   their class is 0 in $H_1(\cK_\BB(f_1,f_2))$,
then by Corolary \ref{Cor1} we have the same answere: $\dim_\KK\HH_{j} (\cG^2_\BB)=(j+1)\nu$.

\vskip 3mm

\subsection{One syzygy is $0$}

\begin{prop} Assume that $(c_{11},c_{12})^t=0\in H_1(\cK_\BB(f_1,f_2))$ and let
$\tau$ be the  dimension over $\KK$ of the submodule of $H_1(\cK_\BB(f_1,f_2))$
generated by $(c_{21},c_{22})^t$, then the associated flags (\ref{flag1}) and (\ref{flag2}) are:

\vskip 2mm
 \begin{equation} \label{ecua1}
  0 = L_0  \subset    L_1  = \ldots = L_\infty  = F_\infty = \ldots = F_0 =  \BB  \end{equation}
 \begin{equation} \label{ecua2}
  0 = L_0^\prime  \subset
  \bigl(\bigl( \begin{smallmatrix} 0 \cr 0\end{smallmatrix} \bigr) :  \bigl( \begin{smallmatrix} c_{21} \cr c_{22}\end{smallmatrix} \bigr)
  \bigr)_{H_1(\cK_\BB(f_1,f_2))}
    = L_1^\prime  = \ldots = L_\infty^\prime = F_\infty^\prime  = \ldots = \ldots = F_1^\prime \subset  F_0^\prime =  \BB, \end{equation}
and for $j\geq 0$ we have $\dim_\KK \HH_j(\cG^2_\BB) = \nu +   j(\nu-\tau)$.
\end{prop}

\begin{proof}
(\ref{ecua1})  and (\ref{ecua2}) follow  from the definitions (\ref{flag1})  and
   (\ref{flag2}).  In this case, we also have (\ref{a1}) and
   $\dim\HH_0(\cG^1_\BB) = \nu$ and $\dim\HH_j(\cG^1_\BB) = 2 \nu$  for $j\geq 1$.
For $j\geq1$ we have
$$\frac{F_{j-1}}{F_{j-1}\cap F_{ 1}}= \frac{\BB}
{\begin{pmatrix}0:
\begin{pmatrix} c_{21} \cr c_{22} \cr \end{pmatrix}\end{pmatrix}}_{H_1(\cK_\BB(f_1,f_2))},$$ which has dimension $\tau$.
Now $L_1^\prime = \BB$ and for $j\geq 1$ we have
$$ L_j =
\begin{pmatrix}0:\begin{pmatrix} c_{21} \cr c_{22} \cr \end{pmatrix} \end{pmatrix}_{H^1(\cK_\BB(f_1,f_2)^*)} \subset \BB,$$ which
has dimension $\mu - \tau$.
We have

$$\frac{\Ann_\BB(L_j \cap  L_1^\prime)}{\Ann_\BB(L_1^\prime)} = \Ann_\BB(L_j) $$
 which has dimension $\tau$.
 Hence in the exact sequences in Theorem \ref{CDos},
we see that the proces to construct the $j+2$ hyperhomology group from the $j$ consists
in imposing $\tau$ restrictions and then we have freedom in
choosing a $(2\nu - \tau)$-dimensional space;
so in all the dimension increases by $2(\nu - \tau)$.
\end{proof}

If we interchange the roles of the 2 syzygies (i.e. interchange the flags) we obtain for this case
\begin{equation}\label{a1a}
\dim_\KK\HH_0(\cG^1_\BB)= \nu,\hskip 0.25cm
\dim_\KK\HH_{ 1}(\cG^1_\BB)=  2\nu - \tau,  \hskip 0.25cm
\dim_\KK\HH_{j}(\cG^1_\BB)=  2\nu - 2 \tau\hskip 0.25cm\hbox{for}\hskip 0.25cm j\geq2,
\end{equation}
and the invariants
\begin{equation} \label{a1b}
\frac{F_{j-1}}{F_{j-1}\cap F_1^\prime} = \frac{\Ann_\BB(L_j^\prime)}{\Ann_\BB(L_1)} = 0.
\end{equation}
Hence, if we construct the $j+2$ hypercohomology classes in this way, there are no restrictions
on the $j$ cycle, and we have freedom in choosing of dimension $2(\nu - \tau)$ again.

\vskip 3mm

\subsection{One syzygy is in the module generated by the other}

If $\TAU{1}=g\TAU{2}$ for some   unit $g\in \BB,$ then  the corresponding flags are identical:
$$0 = L_0 \subset L_1 = (0:\TAU{1})= L_2= \ldots =
  L_\infty  \subset   F_\infty = \cdots = F_0 = \BB $$
 so that the invariantes (\ref{a1b}) are also 0, and the computations (\ref{a1a})
 are also valid in this case, and so we have  again have for the same reasons
 that  $\dim_\KK \HH_j(\cG^2_\BB) = \nu +   j(\nu-\tau)$.
 \vskip 3mm

\begin{prop}
If $\TAU{1}=g\TAU{2}$ for some non unit $g\in \BB,$ then  the corresponding flags are:
$$0 = L_0 \subset L_1 = ( 0: \TAU{1})\subset  L_2 = (L_1:g) \subset\ldots \subset
 L_j = (L_1:g^{j-1}) \subset L_\infty   =   F_\infty = \cdots = F_0 = \BB $$
{\small
$$
0 = L_0^\prime \subset L_1^\prime = \bigl( 0  : \bigl( \begin{smallmatrix} c_{21} \cr c_{22}\end{smallmatrix} \bigr) )_{H_1(\cK_\BB(f_1,f_2))}= \cdots = L_\infty^\prime =  F^\prime_\infty \subset <g^j>+L_1^\prime = F^\prime_j \subset \cdots \subset <g>+L_1^\prime =F^\prime_1 \subset F^\prime_0 = \BB$$}

The flags  $L_j$ and $F^\prime_j$ stabilize when $j=\hbox{min}\{i\in \NN: g^i\in L_1 \}$ and not before.
let $\tau_2$ be the  dimension over $\KK$ of the submodule of $H_1(\cK_\BB(f_1,f_2))$
generated by $(c_{21},c_{22})^t$, then
for $j\geq 0$ we have $\dim_\KK \HH_j(\cG^2_\BB) = \nu +   j(\nu-\tau_2)$.
\end{prop}
\begin{proof} For $j\geq1$ we have
$F_j=\BB$ directly from the assumption, as well as
$L_j^\prime = L_1^\prime$, since $L_1^\prime \bigl( \begin{smallmatrix} c_{11} \cr c_{12}\end{smallmatrix} \bigr) =0$.

We now prove that $F^\prime_j=<g^j>+L'_1$ for $j\geq 1$ by induction.
We have that
$$
a \in F^\prime_1 :=(\TAU{1}: \TAU{2} )_{H_1(\cK_\BB(f_1,f_2))}
= ( g\TAU{2}: \TAU{2} )_{H_1(\cK_\BB(f_1,f_2))}
$$
means that there exists $b\in\BB$ such that
$$(a-bg)\TAU{2}=0 \ \ \Leftrightarrow (a-bg) \ \ \in (0:\TAU{2}) \ \ \Leftrightarrow \ \ a \in <g> + (0:\TAU{2}). $$
We assume that our claim is true for $j$ and prove for $j+1.$
$$
a \in F^\prime_{j+1} :=(F^\prime_j \TAU{1}: \TAU{2} )_{H_1(\cK_\BB(f_1,f_2))}
= (F^\prime_j g\TAU{2}: \TAU{2} )_{H_1(\cK_\BB(f_1,f_2))}
$$
means that there exists $b = c  g^j + m \in F^\prime_j $, with $c\in\BB, \ m \in (0: \TAU{2} ).$ Then,

$$(a-bg)\TAU{2}=0 \Rightarrow  (a-bg) \in (0:\TAU{2}) \Rightarrow (a-c  g^{j+1} + mg) \in (0:\TAU{2}) \Rightarrow a \in <g^{j+1}> + (0:\TAU{2}). $$

On the other hand, if $a \in <g^{j+1}> + (0:\TAU{2}),$
there are $c\in\BB, \ m \in (0: \TAU{2} )$ such that  $a= c  g^{j+1} + m.$ Then
$$
a\TAU{2}=c  g^{j+1}\TAU{2} + m\TAU{2}=c  g^{j}\TAU{1}\in \BB g^j\TAU{1}\subset F_{j}\TAU{1}.
$$
by induction hypothesis.
Therefore, $a\in F_{j+1}.$

Now we  prove $L_j= (L_1:g^{j-1})$ for $j\geq1$ by induction. We note that it is true for $j=1$ by the convention $g^0=1.$
The
assumption implies:
\begin{equation} \label{b1}
<\TAU{1}>\subset<\TAU{2}>  \hskip 1cm,\hskip1cm  L_1^\prime := (0:\TAU{2}) \subset (0:\TAU{1}):=L_1 \subset \BB.
\end{equation}

Let $a \in L_{j+1} :=(L_j\TAU{2}:\TAU{1})=(L_j\TAU{2}:g\TAU{2})$, then $(ag-b) \in (0:\TAU{2})$, with
$b \in (L_1:g^{j-1})$. Multiplying by $g^{j-1}$ we obtain
$$ag^j = bg^{j-1}+(0:\TAU{2}) \subset (0:\TAU{1}) \Rightarrow a \in (L_1:g^j).$$
On the other hand, if $a \in (L_1:g^j)$ we have that $ag\in (L_1:g^{j-1})=L_j.$ Then,
$$
a\TAU{1}=ag\TAU{2}\in L_j\TAU{2}.
$$
Therefore $a\in L_{j+1}$.

Directly from Theorem  \ref{CDos} we obtain
$\dim_\KK \HH_0(\cG^2_\BB) = \nu $ and
$\dim_\KK \HH_1(\cG^2_\BB) = 2 \nu     -\tau_2$.

To  prove the rest, we interchange the roles of the 2 syzygies. We have $\frac{\BB}{F_1}=0$, so the flag induced
from $F^\prime_*$ is $0$. Now $L_1^\prime  \subset L_1 \subset L_j$ from (\ref{b1}), so that
the invariants $\frac{L_1^\prime}{L_j \cap L_1^\prime} = 0$.
Hence the algorithm to construct the (j+2)-cycles from the $j$-cycles imposes no conditions and we have
freedom of choosing $\HH_*(\cG^1_\BB)$, which has dimension $2\nu - 2 \tau_2$.
\end{proof}

gmont@cimat.mx

Centro de Investigaci\'on en Matem\'aticas, AP 402, Guanajuato, 36000, M\'exico

\vskip 3mm
lcn8m@virginia.edu

Department of Mathematics, University of Virginia, Charlottesville, VA 22903, USA

\end{document}